\documentclass[12pt]{article}

\usepackage{amsmath,amsfonts,amssymb,amsthm}
\usepackage{mathrsfs}
\usepackage{latexsym}
\usepackage[english]{babel}
\usepackage{graphicx}
\usepackage[usenames,dvipsnames]{xcolor}

\renewenvironment{proof}[1][\proofname]{{\bfseries #1.} }{\qed}

\def\Cov{{\rm Cov\,}}

\newcommand{\field}[1]{\mathbb{#1}}

\newcommand{\R}{\field{R}}

\newcommand{\Var}{{\rm Var}}
\newcommand{\Corr}{{\rm Corr}}

\def\authors#1{{ \begin{center} #1 \vspace{0pt} \end{center} } \smallskip}
\def\institution#1{{\sl \begin{center} #1 \vspace{0pt} \end{center} } }
\def\inst#1{\unskip $^{#1}$}
\def\title#1{{\huge\bf  \begin{center} #1 \vspace{0pt} \end{center}  } \smallskip}

\def\P{{\mathbb{P}}}

\newtheorem{theorem}{Theorem}[section]

\newtheorem{proposition}[theorem]{Proposition}
\newtheorem{condition}[theorem]{Condition}
\newtheorem{lemma}[theorem]{Lemma}
\newtheorem{corollary}[theorem]{Corollary}
\newtheorem{definition}[theorem]{Definition}

\newtheorem{remark}[theorem]{Remark}

\begin{document}

\date{July 2020}

\title{\sc Lipschitz-Killing Curvatures for\\ Arithmetic Random Waves}
\authors{\large Valentina Cammarota\inst{*}, Domenico Marinucci\inst{\diamond,}\footnote{Corresponding author (e-mail address: \texttt{marinucc@mat.uniroma2.it}).} and Maurizia Rossi\inst{\dag}}
\institution{\inst{*}Dipartimento di Scienze Statistiche, Universit\`a di Roma La Sapienza\\
\inst{\diamond}Dipartimento di Matematica, Università di Roma Tor Vergata\\
\inst{\dag}Dipartimento di Matematica e Applicazioni, Universit\`a di Milano-Bicocca}

\begin{abstract}

In this paper, we show that the Lipschitz-Killing Curvatures for the
excursion sets of Arithmetic Random Waves (toral Gaussian eigenfunctions) are dominated, in the
high-frequency regime, by a single chaotic component.  The latter can be written as a
simple explicit function of the threshold parameter times the centered
norm of these random fields; as a consequence, these geometric functionals
are fully correlated in the high-energy limit. The derived formulae show
a clear analogy with related results on the round unit sphere and suggest the
existence of a general formula for geometric functionals of random
eigenfunctions on Riemannian manifolds.

\smallskip


\smallskip

\end{abstract}

\section{Introduction and general framework}

\subsection{Toral eigenfunctions and Arithmetic Random Waves}

Arithmetic Random Waves (i.e., toral Gaussian eigenfunctions) were introduced nearly
a decade ago in \cite{oravecz, RudnickWigman} and have been investigated very
widely ever since, see for instance \cite{granville,KKW,MPRW2015} and more
recently \cite{BMW,Buckley,Cammarota2018,dalmao,Rudnick,RudnickYesha}; see
also \cite{angstdalmao, angstuniv, bally17} for related
results on random trigonometric polynomials. Interest in their investigation
is motivated both by mathematical physics applications, and by the rich
interplay of probability, geometry and even number theory that characterizes
the behaviour of geometric functionals of their excursion sets.

Let us start
recalling their definition; for an integer $d\ge 2$, let $f:\mathbb{T}^{d}:=\mathbb{R}^{d}/\mathbb{Z}%
^{d}\rightarrow \mathbb{R}$  be the real-valued functions
satisfying the eigenvalue equation
\begin{equation}
\Delta f+Ef=0,  \label{Schrodinger}
\end{equation}%
where $E > 0$ and $\Delta $ is the Laplace-Beltrami operator on $\mathbb{T}%
^{d}$; the spectrum of $\Delta $ is totally discrete. (For $d=2$ we will often write $\mathbb T$ in place of $\mathbb T^2.$) Indeed, the
eigenspaces of the Laplacian on the torus are related to the theory of
lattice points on $(d-1)$-dimensional spheres: let
\begin{equation*}
S:=\{n\in {\mathbb{Z}}:n=n_{1}^{2}+\cdots +n_{d}^{2},\text{ for some }
n_{1},\dots ,n_{d}\in \mathbb{Z}\}
\end{equation*}%
be the collection of all numbers expressible as a sum of $d$ squares. The
sequence of eigenvalues, or \textit{energy levels}, are all numbers of the
form $E_{n}=4\pi ^{2}n,$ $n\in S.$ In order to describe the Laplace
eigenspace corresponding to $E_{n}$, we introduce the set of frequencies $%
\Lambda _{n}$; for $n\in S_{n}$ let
\begin{equation*}
\Lambda _{n}=\{\lambda \in \mathbb{Z}^{d}:\;||\lambda ||^{2}=n\}.
\end{equation*}%
$\Lambda _{n}$ is the frequency set corresponding to $E_{n}$. Using the
notation $e(t):=\exp (2\pi it)$ for $t\in \mathbb R$, the $\mathbb{C}$-eigenspace $\mathcal{E}_{n}$
corresponding to $E_{n}$ is spanned by the $L^{2}$-orthonormal set of
functions $\{e(\langle \lambda ,\cdot \rangle )\}_{\lambda \in \Lambda _{n}}$%
. We denote the dimension of $\mathcal{E}_{n}$
\begin{equation*}
\mathcal{N}_{n}=\text{dim}\,\mathcal{E}_{n}=|\Lambda _{n}|,
\end{equation*}%
that is equal to the number of different ways $n$ may be expressed as a sum
of $d$ squares. In particular, for $d=2$, $\mathcal N_n$ is subject to large and erratic fluctuations; it grows \emph{on average} \cite{La} as $\sqrt{\log n}$, but could be as small as $8$ for an infinite sequence of prime numbers $p \equiv 1\, (\text{mod } 4)$, or as large as a power of $\log n$.

The frequency set $\Lambda _{n}$ can be identified with the set of lattice
points lying on a $(d-1)$-dimensional sphere with radius $\sqrt{n}$, the
sequence of spectral multiplicities $\{\mathcal{N}_{n}\}_{n\in S}$ is
unbounded. It is natural to consider properties of \textit{generic} or
\textit{random} eigenfunctions $f_{n}\in \mathcal{E}_{n}$, in the high-energy asymptotics regime. More precisely, let $f_{n}:\mathbb{T}%
^{d}\rightarrow \mathbb{R}$ be the Gaussian random field of (real valued) $%
\mathcal{E}_{n}$-functions with eigenvalue $E_{n}$, i.e. the random linear
combination
\begin{equation}
f_{n}(x) := \frac{1}{\sqrt{\mathcal{N}_{n}}}\sum_{\lambda \in \Lambda
_{n}}a_{\lambda }e(\langle \lambda ,x\rangle ), \quad x\in \mathbb T^d, \label{fn}
\end{equation}
where the coefficients $\{a_{\lambda }\}_{\lambda \in \Lambda_n, n\in S}$ are complex-Gaussian random
variables\footnote{defined on some probability space $(\Omega, \mathcal F, \mathbb P)$} verifying the following properties:
\begin{enumerate}
\item every $a_{\lambda }$ has the form $a_{\lambda }=\mathrm{Re}(a_{\lambda
})+i\,\mathrm{Im}(a_{\lambda })$ where $\mathrm{Re}(a_{\lambda })$ and $%
\mathrm{Im}(a_{\lambda })$ are two independent real-valued, centred,
Gaussian random variables with variance $1/2$, \label{i}
\item the $a_{\lambda }$'s are stochastically independent, save for the
relations $a_{-\lambda }=\overline{a}_{\lambda }$ in particular making $f_{n}$
real-valued. \label{iii}
\end{enumerate}
By definition, $f_{n}$ is stationary, i.e. the law
of $f_{n}$ is invariant under all translations
\begin{equation*}
f(\cdot )\rightarrow f(y+\cdot ),\hspace{1cm}y\in {\mathbb{T}}^{d};
\end{equation*}%
in fact $f_{n}$ is a centred Gaussian random field with covariance function
\begin{equation*}
\mathbb{E}[f_{n}(x)f_{n}(y)]=\frac{1}{\mathcal{N}_{n}}\sum_{\lambda \in
\Lambda _{n}}e(\langle \lambda ,x-y\rangle),\qquad x,y\in \mathbb T.
\end{equation*}%
Note that the normalizing factor in (\ref{fn}) is chosen so that $f_{n}$ has
unit-variance.

\subsection{Notation}\label{sec-not}

We will use $\lambda, \lambda_1, \lambda_2 \dots$ and in general $\lambda_i$, $i=1,2,\dots$ to denote elements of $\Lambda_n$, while $\lambda_{(\ell)}$ and
$\lambda_{i,(\ell)}$ with $\ell=1, \dots, d$, will denote the $\ell$-th component of
the vectors $\lambda$ and $\lambda_i \in \Lambda_n$ respectively. The
indices $j,\ell$ always run from $1$ to $d$.

For $\ell=1,\dots d$, we denote with $\partial _{\ell}f_{n}(x)$ the derivative of $%
f_{n}(x)$ with respect to $x_{\ell}$. A straightforward differentiation of (\ref{fn}) gives
\begin{equation}\label{der}
\partial _{\ell}f_{n}(x)=\frac{2\pi i}{\sqrt{\mathcal{N}_{n}}}\sum_{\lambda \in
\Lambda _{n}}a_{\lambda }\lambda _{(\ell)}e(\langle \lambda ,x\rangle ),
\end{equation}%
in view of \cite[Lemma 2.3]{Rudnick}, see formula \eqref{RW} below, the
random field $\partial _{\ell}f_{n}$ has variance
\begin{align*}
\text{Var}(\partial _{\ell}f_{n}(x))& =\frac{2^{2}\pi ^{2}(-1)}{\mathcal{N}_{n}}%
\sum_{\lambda _{1},\lambda _{2}\in \Lambda _{n}}\mathbb{E}[a_{\lambda
_{1}}a_{\lambda _{2}}]\,\lambda _{1,(\ell)}\,\lambda _{2,(\ell)}\,e(\langle
\lambda _{1},x\rangle )\,e(\langle \lambda _{2},x\rangle ) \\
& =\frac{2^{2}\pi ^{2}}{\mathcal{N}_{n}}\sum_{\lambda \in \Lambda
_{n}}\lambda _{(\ell)}^{2}=\frac{2^{2}\pi ^{2}n}{d} = \frac{E_n}{d},
\end{align*}%
%
we introduce then the normalized derivative $f_{n,\ell}(x)$ defined by
\begin{equation*}
f_{n,\ell}(x):=\frac{\partial _{\ell}f_{n}(x)}{2\pi \sqrt{\frac{n}{d}}}=i\sqrt{%
\frac{d}{n\mathcal{N}_{n}}}\sum_{\lambda \in \Lambda _{n}}\lambda
_{(\ell)}a_{\lambda }e(\langle \lambda ,x\rangle ).
\end{equation*}%
Note that $f_{n,\ell}(x)$ is real-valued since $f_{n,\ell}^{2}(x)=f_{n,\ell}(x)%
\overline{f_{n,\ell}}(x)$. Analogously, we denote with $\partial^2_{j\ell} f_n$ the second derivative of $f_n$ with respect to $x_\ell$ and $x_j$
\begin{equation}\label{der2}
\partial^2_{j\ell} f_n(x) = -\frac{4\pi^2}{\sqrt{\mathcal N_n}} \sum_{\lambda\in \Lambda_n} a_\lambda \lambda_{(\ell)} \lambda_{(j)} e(\langle \lambda, x\rangle).
\end{equation}
We note that conditions \ref{i}) and \ref{iii}) in %
\eqref{fn} immediately imply that
\begin{equation*}
\mathbb{E}[a_{\lambda }^{2}]=\mathbb{E}[(\mathrm{Re}(a_{\lambda }))^2]-\mathbb{E}[(\mathrm{Im}(a_{\lambda }))^2]=0,
\end{equation*}%
and that $2|a_{\lambda }|^{2}$ has a chi-squared distribution with $2$
degrees of freedom:
\begin{equation*}
\mathbb{E}[|a_{\lambda }|^{2}]=1,\hspace{1cm}\mathbb{E}[(|a_{\lambda
}|^{2}-1)^{2}]=\mathrm{Var}(|a_{\lambda }|^{2})=1,\hspace{1cm}\mathbb{E}%
[|a_{\lambda }|^{4}]=2.
\end{equation*}%
For some of the arguments to follow, we shall make a heavy use of results
and notation recently introduced in the number theory literature by \cite%
{KKW}. In particular, let $\mu _{n}$ be the probability measure on the
circle $\mathcal{S}^{1}:=\{z\in \mathbb{C}:\;||z||=1\}$ defined by
\begin{equation*}
\mu _{n} :=\frac{1}{\mathcal{N}_{n}}\sum_{\lambda \in \Lambda _{n}}\delta _{\lambda/\sqrt n}\text{ ,}
\end{equation*}%
that is to say, the empirical measure for the distribution of integers
corresponding to the eigenvalue $4\pi ^{2}n;$ an important role is going to
be played by the fourth Fourier coefficient of this distribution, i.e.
\begin{equation}
\hat{\mu}_{n}(4):=\int_{\mathcal{S}^{1}}z^{4}d\mu _{n}(z)=\frac{1}{n^{2}%
\mathcal{N}_{n}}\sum_{\lambda \in \Lambda _{n}}(\lambda _{1}+i\lambda
_{2})^{4}.  \label{muquattro}
\end{equation}
This coefficient will not appear in our statements, but inspection of the
proof reveals its important role.

\section{Main results}\label{sec-main}

The purpose of this paper is to provide a full characterization for the
asymptotic behaviour in the high-energy limit of Lipschitz-Killing Curvatures computed on the
excursion sets of two-dimensional Arithmetic Random Waves defined in (\ref{fn}) for $d=2$, and to compare the
results with those recently derived in the case of random spherical
eigenfunctions in \cite{CM2018}; see also \cite{adlertaylor, azaiswschebor, chengxiaob, estradeleon, feng, MRW, npr, PV20} and the references therein for background material and a
number of recent results on Lipschitz-Killing Curvatures in Euclidean
settings or on the unit round sphere.

It is well-known (see e.g., \cite{KKW}) that there exists a density-$1$ subsequence $\lbrace n_j\rbrace_j\subset S$ such that for every $x,y\in \mathbb T^2$
\begin{equation}\label{density1}
\mathbb E[f_n(x/\sqrt{n_j})f_n(y/\sqrt{n_j}) ] \to J_0(2\pi\| x-y\|),\qquad n_j\to +\infty,
\end{equation}
where $J_0$ denotes the Bessel function of the first kind of order zero. %
To fix notation, let us recall first that the excursion sets of $f_n$ are defined by
\begin{equation*}
A_{u}(f_{n};\mathbb{T}):=\{x\in \mathbb{T}:f_{n}(x)\geq u\},\qquad u\in \mathbb R.
\end{equation*}
In the two-dimensional case, it is well-known that the three Lipschitz-Killing
Curvatures $\mathcal{L}_{k}$, $k=0,1,2$ correspond to the area functional $%
k=2$, half the boundary length $k=1$ and the Euler-Poincar\'{e}
Characteristic $k=0$, i.e., the number of connected components minus
the number of ``holes". We are interested in these geometric functionals evaluated at excursion sets of Arithmetic Random Waves: for $u\in \mathbb R$,
\begin{equation}\label{es}
\mathcal L_k(n;u) := \mathcal{L}_{k}(A_{u}(f_{n};\mathbb{T})),\qquad k=0,1,2;
\end{equation}
in particular in their asymptotic behavior as $n\to +\infty$ such that $\mathcal N_n\to +\infty$. To the best of our knowledge, all results concerning $\mathcal L_0(n;u)$, i.e. the Euler-Poincar\'e characteristic, are new.

The expected values of (\ref{es}) are given in the following lemma (see Appendix \ref{Appendixexp}). As usual, we use $\phi$, $\Phi$ to denote
the standard Gaussian density and distribution function, respectively.
\begin{lemma}
\label{expval} The expected values for the Lipschitz-Killing Curvatures on
excursion sets of Arithmetic Random Waves are given by
\begin{equation}\label{mean1}
\mathbb E[\mathcal{L}_{k}(n;u)] = m_k(u) \left (\sqrt{\frac{E_n}{2}} \right )^{2-k},\qquad k=0,1,2
\end{equation}
for $n\in S$ and $u\in \mathbb R$, where
\begin{equation}\label{mk}
m_2(u) := 1-\Phi (u),\quad  m_1(u) := \sqrt{\frac{\pi}{8}}\phi (u),\quad  m_0(u) := \frac{1}{2\pi}u\phi (u).
\end{equation}
\end{lemma}
The next Theorem, which is the main result of this paper, provides a full
characterization of asymptotic fluctuations (in the high-energy limit)
around these expected values. Let us introduce the following notation for
centred functionals: for $n\in S$ and $u\in \mathbb R$
\begin{equation*}
\mathcal{\overline{L}}_{k}(n;u):=\mathcal{L}%
_{k}(n;u)-\mathbb{E}[\mathcal{L}_{k}(n;u)], \qquad k=0,1,2;
\end{equation*}
we will need also the following subset of the set of frequencies: if $n$ is not a square we set
$$
\Lambda_n^+ := \lbrace \lambda \in \Lambda_n : \lambda_{(2)} >0\rbrace,
$$
otherwise $\Lambda_n^+ := \lbrace \lambda \in \Lambda_n : \lambda_{(2)} >0\rbrace \cup \lbrace (\sqrt n, 0)\rbrace$. Note that, for every $n\in S$, $\lbrace a_\lambda\rbrace_{\lambda \in \Lambda_n^+}$ are i.i.d. random variables, and $|\Lambda_n^+| = \mathcal N_n/2$. \\
In order to state our main results, we will need some more notation. Let $Q_{0,n}:= [0, 1/\sqrt{E_n})^2$, and denote by $\mathcal L_0(u; Q_{0,n})$ the Euler-Poincar\'e characteristic of the intersection between the excursion set $\lbrace f_n \ge u\rbrace$ and the square $Q_{0,n}$.
 \begin{condition}\label{condizione}
For $n\in S$
\begin{equation}\label{cond2}
\mathbb E[\mathcal L_0(u; Q_{0,n})(\mathcal L_0(u; Q_{0,n})-1)] = O(1),
\end{equation}
where the constant involved in the $O$-notation is absolute.
\end{condition}
Note that Condition \ref{condizione} only concerns the zero-th Lipschitz-Killing curvature.
\begin{remark}\rm
The estimate (\ref{cond2}) holds for a density-$1$ subsequence of eigenvalues in the high-energy limit (thanks to (\ref{density1}) and \cite{CMW16}), and
we do believe (\ref{cond2}) to be true for every $n\in S$.
\end{remark}
\begin{theorem}
\label{mainterm} For $k=0,1,2$, $n\in S$  it holds that
\begin{equation}\label{simple}
\overline{\mathcal L}_k(n;u) = \frac{c_k(u)}{\sqrt{\mathcal N_n/2}} \left (\sqrt{\frac{E_n}{2}} \right )^{2-k}   \frac{1}{\sqrt{\mathcal{N}_{n}/2} }\sum_{\lambda\in \Lambda^+_n }(|a_{\lambda }|^{2}-1) + \mathcal R_k(n;u),
\end{equation}
where
\begin{equation}\label{ck}
c_2(u) := \frac12 u\phi(u), \quad c_1(u) := \frac{1}{2}\sqrt{\frac{\pi }{8}}u^{2}\phi (u),\quad c_0(u) := \frac{1}{2} (u^{2}-1)u\phi (u) \frac{1}{2\pi },
\end{equation}
and under Condition \ref{condizione}
\begin{equation}\label{err}
\mathbb E[\mathcal R_k(n;u)^2] = O\left ( \frac{E_n^{2-k}}{\mathcal N_n^2}\right ),
\end{equation}
the constant involved in the $O$-notation only depending on $k$.
\end{theorem}
In particular, Theorem \ref{mainterm} (whose proof will be given in \S \ref{proofs}) shows that the ``first order approximation" of $\mathcal L_k(n;u)$ for any $k$ can be written as a simple explicit function (depending on $k$) of the threshold parameter $u$ times the centered norm of $f_n$. Indeed,
\begin{equation}\label{norm1}
\| f_n\|^2_{L^2(\mathbb T)} - \mathbb E\left [ \| f_n\|^2_{L^2(\mathbb T)}\right ] = \frac{1}{\mathcal N_n} \sum_{\lambda\in \Lambda_n} (|a_\lambda|^2 -1) = \frac{1}{\mathcal N_n/2} \sum_{\lambda\in \Lambda_n^+} (|a_\lambda|^2 -1),
\end{equation}
cf. (\ref{simple}). This has several important consequences, as discussed just below and in \S \ref{sec-disc}.
Moreover, for any $u\in \mathbb R$, $k=0,1,2$, as $n\to +\infty$ such that $\mathcal N_n\to +\infty$,
\begin{equation}\label{var-asymp}
\Var(\mathcal{{L}}_{k}(n;u)) = \frac{c_k(u)^2}{2^{1-k}} \frac{E_n^{2-k}}{\mathcal N_n}
 + O\left ( \frac{E_n^{2-k}}{\mathcal N_n^2}\right ),
 \end{equation}
where the constant involved in the $O$-notation only depends on $u$ and $k$. Let us now define
$\mathcal U_k:= \lbrace 0\rbrace$ for $k=1,2$ and $\mathcal U_0 :=\lbrace -1, 0,1\rbrace$; note that
$c_k(u)$ defined in (\ref{ck}) vanishes if and only if $u\in \mathcal U_k$.  An easy by-product of Theorem \ref{mainterm} is the following quantitative Central Limit Theorem in
Wasserstein distance\footnote{Given $X$, $Y$ integrable random variables, $
d_W(X,Y) :=\sup_{h\in Lip(1)} \left | \mathbb E[h(X)] - \mathbb E[h(Y)] \right |,
$
where $Lip(1)$ denotes the space of Lipschitz functions $h:\mathbb R\to \mathbb R$ whose Lipschitz constant is $\le 1$.} (written $d_W$).
\begin{corollary}\label{cor1}
As $n\rightarrow \infty $ such that $\mathcal N_n\to +\infty$, for $k=0,1,2$,  $u\notin \mathcal U_k$ under Condition \ref{condizione}
\begin{equation*}
d_{W}\left( \widetilde{\mathcal L}_k(n;u),Z \right) = O\left (  \frac{1}{\sqrt{\mathcal N_n}}\right )
\end{equation*}
where
$$
\widetilde{\mathcal L}_k(n;u) := \frac{\overline{\mathcal{L}}_{k}(n;u)}{\sqrt{\Var(\mathcal{L}
_{k}(n;u))}},
$$
$Z\sim \mathcal{N}(0,1)$, and the constant involved in the $O$-notation only depends on  $k$.
\end{corollary}
The proof of Corollary \ref{cor1} will be given in \S \ref{proofs}.
Theorem \ref{mainterm} also allows to deduce Moderate Deviation estimates \cite[\S 1.2]{DZ98} for Lipschitz-Killing curvatures evaluated at excursion sets of Arithmetic Random Waves, see also \cite[Remark 1.9]{MRT20}.
\begin{corollary}\label{cor2} For $k=0,1,2$, $n\in S$, $u\notin \mathcal U_k$, let $\lbrace s^{(k)}_{n;u}\rbrace_{n\in S}$ be any sequence of positive numbers such that as $\mathcal N_n\to +\infty$
\begin{equation}\label{cond1}
s^{(k)}_{n;u}\to +\infty, \qquad \qquad \frac{s^{(k)}_{n;u}}{\sqrt{\log  \mathcal N_n}} \to 0.
\end{equation}
Under Condition \ref{condizione} the sequence of random variables
\begin{equation*}
\left \lbrace \widetilde{\mathcal L}_k(n;u)/ s_{n;u}^{(k)} \right \rbrace_{n\in S}
\end{equation*}
satisfies a Moderate Deviation principle with speed $(s_{n;u}^{(k)})^2$ and rate function $\mathcal I(x):= x^2/2$, $x\in \mathbb R$, i.e. for any Borelian set $B\subset \mathbb R$
\begin{align*}
-\inf_{x\in \mathring B} \mathcal I(x) &\le \liminf_{\mathcal N_n\to +\infty} \frac{1}{(s^{(k)}_{n;u})^2}\log \mathbb P \left ( \frac{\widetilde{\mathcal L}_k(n;u)}{s^{(k)}_{n;u}}\in B \right )\cr
& \le \limsup_{\mathcal N_n\to +\infty} \frac{1}{(s^{(k)}_{n;u})^2}\log \mathbb P \left ( \frac{\widetilde{\mathcal L}_k(n;u)}{s^{(k)}_{n;u}}\in B \right )\le -\inf_{x\in \bar B} \mathcal I(x),
\end{align*}
where $\mathring B$ (resp. $\bar B$) denotes the interior (resp. the closure) of $B$.
\end{corollary}
Corollary \ref{cor2} is a refinement of the Central Limit Theorem in Corollary \ref{cor1}, its proof will be given in \S \ref{proofs}. A further obvious consequence of Theorem \ref{mainterm} is the following asymptotic full correlation result.\begin{corollary}\label{cor3}
Let $k_{1},k_{2}\in \lbrace 0,1,2\rbrace$ and $u_1, u_2\notin \mathcal U_0$. As $n\rightarrow \infty$ such that $\mathcal N_n\to +\infty$, under Condition \ref{condizione},
 \begin{equation*}
\Corr\left( \mathcal{L}_{k_{1}}(n;u_1),\mathcal{L}_{k_{2}}(n;u_2)\right) = 1 + O\left ( \frac{1}{\sqrt{\mathcal N_n}}\right ),
\end{equation*}
where the constant involved in the $O$-notation only depends on $k_1$ and $k_2$.
\end{corollary}
In words, Corollary \ref{cor3} (whose proof will be given in \S \ref{proofs}) entails that in the
``nondegenerate" points where the leading term in the asymptotic variance (\ref{var-asymp})  does not vanish knowledge of one of the
three Lipschitz-Killing curvatures at some level allows the derivation of the other two at any level,
up to a term which is lower order in the $L^{2}(\mathbb P )$-sense.

See \S \ref{sec-disc} for further comments on our main result, its consequences and the comparison with the spherical case.
\begin{remark}[Nodal case]\label{rem1}\rm
The geometry of Arithmetic Random Waves was initially investigated in \cite{oravecz, RudnickWigman} and subsequently in several works with a focus on the nodal case which corresponds to the level $u=0$. Concerning the (half) nodal length, the asymptotic variance was addressed and fully solved in \cite{KKW}: as $\mathcal N_n\to +\infty$,
\begin{equation}\label{var0}
\Var(\mathcal L_1(n;0)) = \frac{1}{4} \cdot \frac{1+\hat{\mu}_n(4)^2}{512}\frac{E_n}{\mathcal N_n^2} \left ( 1+o(1)\right ),
\end{equation}
where $\hat{\mu}_n(4)$ has been defined in (\ref{muquattro}). It is well-known that for any $\mu\in [-1,1]$, there exists a sequence of energy levels such that the corresponding sequence of fourth Fourier coefficients converges to $\mu$. The second order fluctuations of the nodal length were investigated in \cite{MPRW2015}: as $\mathcal N_n\to +\infty$ and $\hat{\mu}_n(4)\to \mu$
\begin{equation}\label{dist0}
\widetilde{\mathcal L}_1(n;0) := \frac{\overline{\mathcal L}_1(n;0)}{\sqrt{\Var(\mathcal L_1(n;0))}} \mathop{\to}^d \frac{1}{2\sqrt{1+\mu^2}}(2-(1+\mu)Z_1^2 + (1-\mu)Z_2^2),
\end{equation}
where $Z_1,Z_2$ are i.i.d. standard Gaussian random variables. A quantitative Limit Theorem in Wasserstein distance is given in \cite{PR18}.

The ``signed area" of Arithmetic Random Waves restricted to shrinking balls of radius above the Planck scale has been recently investigated in \cite{KWY20}, the Euler-Poincar\'e characteristic of the excursion set at level zero is currently under investigation.
\end{remark}

\section{Outline of the paper}\label{sec-disc}

\subsection{On the proofs}

Our approach to proving the results of this paper stated in \S \ref{sec-main} is broadly analogous to
what was used earlier to evaluate the Lipschitz-Killing Curvatures of
excursion sets for random eigenfunctions on the sphere or in Euclidean
settings, see e.g. \cite{KratzLeon, DI, MW2014, MR2015, estradeleon, CM2018, npr, dalmao, PV20} and the references therein.
The starting point is to derive the so-called chaotic decomposition of
our geometric functionals, that is, for $k=0,1,2$, $n\in S$ and $u\in \mathbb R$, a series expansion in $L^2(\mathbb P)$ of the form
\begin{equation}\label{ce1}
\mathcal{L}_{k}(n;u)=\sum_{q=0}^{\infty }\text{Proj%
}[\mathcal{L}_{k}(n;u)|q],
\end{equation}%
where $\text{Proj}[\mathcal{L}_{k}(n;u)|q]$
denotes the orthogonal projection of $\mathcal{L}_{k}(n;u)$
on the space spanned by \emph{multivariate} Hermite polynomials of order $q$ to be computed on
 $f_n$ in (\ref{fn}) and their derivatives up to order two, and the random variables $\text{Proj}[\mathcal{L}_{k}(n;u)|q]$ and $\text{Proj}[\mathcal{L}_{k}(n;u)|q']$ are orthogonal whenever $q\ne q'$. Recall that Hermite polynomials $\lbrace H_q\rbrace_{q\in \mathbb N}$ are defined as
\begin{equation}\label{herm}
H_0 \equiv 1,\qquad  H_q(t) := (-1)^q \phi^{-1}(t) \frac{d^q}{dt^q} \phi(t), \ t\in \mathbb R, \ q\ge 1,
\end{equation}
where $\phi$ still denotes the standard Gaussian probability function.
A more complete discussion on Wiener chaos is
given in \S \ref{sec-chaos}, see also \cite[\S 2]{noupebook} for background and details.

The zero-th order term just amounts to the expected value of Lipschitz-Killing curvatures given in Lemma \ref{expval},
whereas the first order projection is easily seen to vanish
identically due to the oscillating properties of these random waves $f_n$.
\begin{lemma}\label{lem01}
For $k=0,1,2$, $n\in S$ and $u\in \mathbb R$
\begin{equation}
\text{\rm Proj}[\mathcal{L}_{k}(n;u)|0] =  m_k(u) \left (\sqrt{\frac{E_n}{2}} \right )^{2-k},
\end{equation}
where $m_k(u)$ are as in (\ref{mk}),
moreover for $n\ge 1$
\begin{equation}\label{ca1}
\text{\rm Proj}[\mathcal{L}_{k}(n;u)|1] = 0.
\end{equation}
\end{lemma}
It becomes
then crucial to investigate the behaviour of $\text{Proj}[\mathcal{L}%
_{k}(n;u)|2]$.
\begin{proposition}
\label{secondchaosb}
For $k=0,1,2$, $n\in S$ and $u\in \mathbb R$ it holds that
\begin{equation}\label{secondchaosb1}
\text{\rm Proj}[\mathcal{L}_{k}(n;u)|2] = \frac{c_k(u)}{\sqrt{\mathcal N_n/2}} \left (\sqrt{\frac{E_n}{2}} \right )^{2-k}   \frac{1}{\sqrt{\mathcal{N}_{n}/2} }\sum_{\lambda\in \Lambda^+_n }(|a_{\lambda }|^{2}-1),
\end{equation}
where, as in (\ref{ck}),
\begin{equation*}
c_2(u) = \frac12 u\phi(u), \quad c_1(u) = \frac{1}{2}\sqrt{\frac{\pi }{8}}u^{2}\phi (u),\quad c_0(u) = \frac{1}{2} (u^{2}-1)u\phi (u) \frac{1}{2\pi }.
\end{equation*}
\end{proposition}
From (\ref{secondchaosb1}) it is clear that whenever $c_k(u)\ne 0$ the variance of the second order chaotic projection of $\mathcal L_k(n;u)$ is of order $E_n^{2-k}/\mathcal N_n$, otherwise $\text{Proj}[\mathcal{L}_{k}(n;u)|2]=0$.
A careful investigation of higher order chaotic components 
yields the following.
\begin{proposition}\label{prop-ho}
For $k=0,1,2$, $n\in S$, $u\in \mathbb R$, as $n\to +\infty$ such that $\mathcal N_n\to +\infty$,
under Condition \ref{condizione}
 \begin{equation}\label{hcv}
\Var \left (\sum_{q=3}^{+\infty}\mathrm{Proj}[\mathcal{L}_{k}(A_{u}(f_{n};\mathbb{T}))|q] \right ) = O\left ( \frac{E_n^{2-k}}{\mathcal N_n^2} \right ),
\end{equation}
where the constant involved in the $O$-notation only depends on $k$.
\end{proposition}
Proposition \ref{prop-ho} together with (\ref{secondchaosb1}) proves Theorem \ref{mainterm}, see \S \ref{proofs}, once setting
$$
\mathcal R_k(n;u) := \sum_{q=3}^{+\infty}\mathrm{Proj}[\mathcal{L}_{k}(A_{u}(f_{n};\mathbb{T}))|q].
$$
The proofs of Corollary \ref{cor1}, Corollary \ref{cor2} and Corollary \ref{cor3} (postponed to \S \ref{proofs}) heavily rely on (\ref{simple}), i.e. on the fact that, at least at ``non-degenerate" levels, all Lipschitz-Killing Curvatures for Arithmetic Random Waves behave (in the high-energy limit) as an element of a fixed order Wiener chaos, in particular as a sum of i.i.d. random variables. Equation (\ref{err}) quantifies the error made when replacing $\overline{\mathcal L}_k(n;u)$ with the empirical mean of centered squared Fourier coefficients $\lbrace a_\lambda\rbrace_{\lambda\in \Lambda_n}$, allowing to get the quantitative estimates  stated in these corollaries.

\subsection{Discussion}

The first few Hermite polynomials  (\ref{herm})
$$H_{0}(u)=1,\quad H_{1}(u)=u,\quad
H_{2}(u)=u^{2}-1,\qquad u\in \mathbb R$$
will play a crucial role in the arguments to
follow. For the sake of notational simplicity, let us set
$$
H_{-1}(u) := 1 -\Phi(u), \qquad u\in \mathbb R.
$$
For the moment, let us observe that we can rewrite the empirical mean on the right hand side of (\ref{secondchaosb1}) as
\begin{equation}\label{oss1}
\frac{1}{\mathcal{N}_{n}}\sum_{\lambda\in \Lambda_n
}(|a_{\lambda }|^{2}-1)  = \int_{\mathbb{T}
}H_{2}(f_{n}(x))\,dx.
\end{equation}%
Of course,
\begin{equation}\label{oss2}
1 = \int_{\mathbb T} H_0(f_n(x))\,dx.
\end{equation}
From (\ref{oss2}), bearing in mind Lemma \ref{lem01}, we can rewrite (\ref{mean1}) as
\begin{equation}\label{toro0}
\begin{split}
&\text{\rm Proj}[\mathcal{L}_{k}(n;u)|0] = m_k(u) \left (\sqrt{\frac{E_n}{2}} \right )^{2-k}\int_{\mathbb T} H_0(f_n(x))\,dx,\\
&m_2(u) = H_{-1}(u),\quad  m_1(u) = \sqrt{\frac{\pi}{8}}H_0(u)\phi (u),\quad  m_0(u) = \frac{1}{2\pi}H_1(u)\phi (u),
\end{split}
\end{equation}
for every $n\in S$, $k=0,1,2$ and $u\in \mathbb R$. Analogously, from (\ref{oss1}), we can rewrite (\ref{secondchaosb1}) as
\begin{equation}\label{toro2}
\begin{split}
&\text{\rm Proj}[\mathcal{L}_{k}(n;u)|2] = c_k(u) \left (\sqrt{\frac{E_n}{2}} \right )^{2-k}  \int_{\mathbb T} H_2(f_n(x))\,dx,\\
&c_2(u) = \frac12 H_1(u)\phi(u), \quad c_1(u) = \frac{1}{2}\sqrt{\frac{\pi }{8}}H_1(u)^2\phi (u),\quad c_0(u) = \frac{1}{2} H_1(u)H_2(u)\phi (u) \frac{1}{2\pi },
\end{split}
\end{equation}
for every $k=0,1,2$, $n\in S$ and $u\in \mathbb R$. In the next subsection, we will show that
(\ref{toro0}) and (\ref{toro2}) are in \emph{perfect} analogy with the spherical case, suggesting the existence of a ``second order" Gaussian Kinematic formula à la Adler and Taylor \cite{adlertaylor}. The generalization of these
expressions to arbitrary dimensions are currently under investigation.

\subsubsection{Random Spherical Harmonics: previous work}\label{subsubsphere}

Random Spherical Harmonics are Gaussian eigenfunctions of the spherical Laplacian operator, that is, the sequence $\lbrace f_\ell\rbrace_{\ell\in \mathbb N}$ of zero mean, unit variance isotropic Gaussian fields on the unit round sphere $\mathbb S^2$ which
satisfies the Helmholtz equation
\begin{equation*}
\Delta _{\mathbb{S}^{2}}f_{\ell }=-\lambda _{\ell }f_{\ell },\qquad
\lambda _{\ell }=\ell (\ell +1),\  \ell \in \mathbb N.
\end{equation*}%
Here, $\Delta_{\mathbb S^2}$ is the spherical Laplacian operator.
For the excursion sets of these fields
$$
A_u(f_\ell;\mathbb S^2) := \lbrace x\in \mathbb S^2 : f_\ell(x) \ge u\rbrace, \qquad u\in \mathbb R,
$$
 the following results hold\footnote{Note that $\int_{\mathbb S^2} H_0(f_\ell(x))\,dx = 4\pi$.} (see \cite%
{CM2018} and the references therein):
\begin{equation*}
\text{\rm Proj}[\mathcal L_k(A_u(f_\ell;\mathbb S^2)) | 0 ] = m_k(u) \left (\sqrt{\frac{\lambda_\ell}{2}} \right )^{2-k}\int_{\mathbb S^2} H_0(f_\ell(x))\,dx + 2H_{-1}(u)\cdot \delta_k^0,
\end{equation*}
where the coefficients $m_k(u)$ are as in (\ref{toro0}), and
\begin{equation}\label{sfera2}
\text{\rm Proj}[\mathcal L_k(A_u(f_\ell;\mathbb S^2)) | 2 ] = c_k(u) \left (\sqrt{\frac{\lambda_\ell}{2}} \right )^{2-k}\int_{\mathbb S^2} H_2(f_\ell(x))\,dx + O_{L^2(\mathbb P)}(1)\cdot \delta_{k}^{0},
\end{equation}
where $c_k(u)$ are as in (\ref{toro2}). Here, $\delta_i^j$ is the Kronecker delta, and $O_{L^2(\mathbb P)}(1)$ stands for a sequence of random variables bounded in $L^2(\mathbb P)$.

\subsubsection{Some more remarks}

\begin{remark}
[On Berry's Cancellation]\rm It should be noted that for all three
Lipschitz-Killing Curvatures the second-order chaos term disappears in the
\emph{nodal} case $u=0$, see (\ref{toro2}) and (\ref{sfera2}). As a result, the asymptotic variance of these geometric functionals is of smaller order for this
value, thus providing an interpretation of the Berry's cancellation
phenomenon first noted (for the case of boundary lengths of planar random
eigenfunctions) in \cite{Berry2002} and then discussed by \cite{Wig,
MPRW2015, CM2018} and others.
\end{remark}

\begin{remark}[On Universality]\rm It was shown in \cite{KKW} and later in \cite{MPRW2015}, see also Remark \ref{rem1},
that some geometric features for the excursion sets of Arithmetic Random
Waves are not universal, in the sense that they do not share a unique limit
as the eigenvalues $n$ diverge. In particular, for the case of nodal length
it turns out that both the variance (\ref{var0}) and the limiting distributions (\ref{dist0}) can vary
quite substantially along subsequences characterized by different limiting
values for the coefficients $\widehat{\mu }_{n}(4), n\in S$ introduced in (\ref
{muquattro}). This is not the case for Lipschitz-Killing Curvatures computed
on excursion sets corresponding to a non-vanishing second-order chaotic component, see \S \ref{sec-main}: their
expected values, their asymptotic variances and their (Gaussian) limiting
distributions are universal (in other words, they are invariant under
subsequences of energy levels). The sequence of parameters $\widehat{\mu }_{n}(4), n\in S$
appears ubiquitously in the proofs that will be presented in the following
section.
\end{remark}

\begin{remark}[On Correlation]\rm Corollary \ref{cor3} is,
heuristically, a consequence of the fact that the fluctuations of
all three Lipschitz Killing Curvatures are actually dominated by their centered norm, see Theorem \ref{mainterm}, (\ref{norm1}) and (\ref{toro2}),
\begin{equation*}
\left\Vert f_{n}\right\Vert _{L^{2}(\mathbb{T})}^{2}-\mathbb{E}\left[
\left\Vert f_{n}\right\Vert _{L^{2}(\mathbb{T})}^{2}\right] =\int_{%
\mathbb{T}}H_{2}(f_{n}(x))\,dx=\frac{1}{\mathcal{N}_{n}}\sum_{\lambda\in \Lambda_n
}(|a_{\lambda }|^{2}-1).
\end{equation*}%
Again, an analogous phenomenon occurs for random eigenfunctions on the
sphere, see \S \ref{subsubsphere}: in the case of Arithmetic Random Waves, however, this is slightly
more surprising, because isotropy does not hold and hence one could expect
the magnitude of the single random coefficients $\left\{ |a_{\lambda
}|^{2}\right\}_{\lambda\in \Lambda_n} $ to play a more relevant role.
\end{remark}

\subsection{Plan}

In \S \ref{sec-chaos} we will recall basic facts on Wiener chaos, the chaotic expansion (\ref{ce1}) for Lispchitz-Killing Curvatures will be stated in \S \ref{sec-ce} while the proofs of our main results will be given in \S \ref{sub-proofs}. In particular, in \S \ref{sec-ce0} we will establish the chaotic expansion for Euler-Poincar\'e characteristic of Arithmetic Random Waves. The second chaotic components of these geometric functionals will be analyzed in \S \ref{sec-2c} leading to the proof of Proposition \ref{secondchaosb}. We will investigate higher order chaotic components in \S \ref{sec-ho} proving Proposition \ref{prop-ho} along the way. Some technicalities and several tedious computations will be collected in the four Appendixes \ref{app-approx}--\ref{Appendix2chaos}.

\subsection*{Acknowledgements}

The research of VC has been supported by the Università di Roma La Sapienza project \emph{Geometry of random fields}.
DM acknowledges support from the MIUR Excellence Project awarded to the
Department of Mathematics, Università di Roma Tor Vergata, CUP
E83C18000100006.
The research of MR has been supported by the ANR-17-CE40-0008 project \emph{Unirandom} and the GNAMPA-INdAM project 2019 \emph{Propriet\'a analitiche e geometriche di campi aleatori}.

\section{Proofs of the main results}\label{proofs}

\subsection{Wiener chaos}\label{sec-chaos}

Let us recall some basic facts on Wiener chaos, restricting ourselves to the toral setting. Bear in mind the definition of Hermite polynomials in (\ref{herm}).
The family $\mathbb{H} := \{H_k/\sqrt{k!}\}_{k\in \mathbb N}$ is a complete orthonormal system in $L^2(\mathbb{R}, \mathscr{B}(\mathbb{R}), \phi(t)dt) =:L^2(\phi),$ where $\phi$ still denotes the standard Gaussian density on the real line.

The Arithmetic Random Waves (\ref{fn}) considered in this work are a by-product of a family of complex-valued Gaussian random variables $\{a_\lambda\}_{\lambda\in \mathbb{Z}^2}$, defined on  $(\Omega, \mathscr{F}, \mathbb{P})$ and satisfying properties (\ref{i}) and (\ref{iii}) in (\ref{fn}).  Let us define the space ${\mathcal A}$ to be the closure in $L^2(\mathbb{P})$ of all real finite linear combinations of random variables $\xi$ of the form $$\xi = z \, a_\lambda + \overline{z} \, a_{-\lambda},$$ where $\lambda\in \mathbb{Z}^2$ and $z\in \mathbb{C}$. The space ${\mathcal A}$ is a real centered Gaussian Hilbert subspace
of $L^2(\mathbb{P})$.
\begin{definition}\label{d:chaos}For $q\in \mathbb N$, the $q$-th Wiener chaos associated with ${\mathcal A}$, written $C_q$, is the closure in $L^2(\mathbb{P})$ of all real finite linear combinations of random variables of the form
$$
H_{p_1}(\xi_1)\cdot H_{p_2}(\xi_2)\cdots H_{p_k}(\xi_k)
$$
for $k\in \mathbb N_{\ge 1}$, where $p_1,...,p_k \in \mathbb N$ satisfy $p_1+\cdots+p_k = q$, and $(\xi_1,...,\xi_k)$ is a standard real Gaussian vector extracted
from ${\mathcal A}$ (in particular, $C_0 = \mathbb{R}$).
\end{definition}
Using the orthonormality and completeness of $\mathbb{H}$ in $L^2(\phi)$, together with a standard monotone class argument (see e.g. \cite[Theorem 2.2.4]{noupebook}), it is easy to show that $C_q$ and $C_m$ are orthogonal in the sense of $L^2(\mathbb{P})$ for every $q\neq m$, and moreover
\begin{equation*}
L^2_{\mathcal A}(\mathbb P):=L^2(\Omega, \sigma({\mathcal A}), \mathbb{P}) = \bigoplus_{q=0}^\infty C_q;
\end{equation*}
that is, every real-valued functional $F$ of ${\mathcal A}$ can be (uniquely) represented in the form
\begin{equation}\label{e:chaos2}
F = \sum_{q=0}^\infty \text{\rm Proj}[F | q]
\end{equation}
where $\text{\rm Proj}[F | q]$
 stands for the projection of $F$ onto $C_q$, and the series converges in $L^2(\mathbb{P})$. Plainly, $\text{\rm Proj}[F | 0]= \mathbb E[F]$.

From (\ref{der}), for $j,\ell = 1,\dots, d$ the random fields $f_{n},\partial_{j} f_n,\partial^2_{j\ell} f_n$ viewed as collections of
Gaussian random variables
indexed by $x\in \mathbb T^d$ are all lying in ${\mathcal A}$, i.e. for every $x\in \mathbb T^d$ we have
\begin{equation*}
f_{n}(x),\, \partial_{j}f_{n}(x), \, \partial^2_{j\ell}f_{n}(x) \in \mathcal A.
\end{equation*}

\subsection{Chaotic expansions of Lipschitz-Killing curvatures}\label{sec-ce}

The three geometric functionals of our interest are finite-variance functionals of $\mathcal A$, hence applying (\ref{e:chaos2}) we get the series expansion in (\ref{ce1}). Let us be more precise.

\subsubsection{Excursion area}

For the second Lipschitz-Killing curvature we have the following integral representation
\begin{equation}
\mathcal L_2(n;u) = \int_{\mathbb T} \mathbf{1}_{\lbrace f_n(x)\ge u\rbrace}\,dx
\end{equation}
entailing that $\mathcal L_2(n;u)\in L^2_{\mathcal A}(\mathbb P)$. The proof of the following result is simple (see also \S 3 in \cite{DI}) and hence omitted.
\begin{lemma}\label{lem-ce2}
For every $n\in S$ and $u\in \mathbb R$, the chaotic decomposition of $\mathcal L_2(n;u)$ is
\begin{equation}
\mathcal L_2(n;u) = \sum_{q=0}^{+\infty} \frac{\gamma_q(u)}{q!} \int_{\mathbb T} H_q(f_n(x))\,dx,
\end{equation}
where $\gamma_q(u):= H_{q-1}(u)\phi(u)$, and the convergence of the series is in $L^2(\mathbb P)$.
\end{lemma}

\subsubsection{Boundary length}

For the first Lipschitz-Killing curvature we have the following \emph{formal} integral representation
\begin{equation}\label{ir1}
\mathcal L_1(n;u) = \frac12 \int_{\mathbb T} \delta_u(f_n(x)) \| \nabla f_n(x)\|\,dx,
\end{equation}
where $\delta_u$ is the Dirac mass in $u$, and $\nabla f_n$ is the gradient of $f_n$. For $\epsilon>0$, let us consider the $\epsilon$-approximating random variable
$$
\mathcal L_1^\epsilon(n;u) := \frac12 \int_{\mathbb T} \frac{1}{2\epsilon}\mathbf{1}_{[u-\epsilon, u+\epsilon]}(f_n(x))\| \nabla f_n(x)\|\,dx.
$$
\begin{lemma}\label{lem-approx1}
For every $n\in S$ and $u\in \mathbb R$ it holds that, as $\epsilon\to 0$,
\begin{equation}
\mathcal L_1^\epsilon(n;u) \to  \mathcal L_1(n;u),
\end{equation}
where the convergence holds both a.s. and in $L^2_{\mathcal A}(\mathbb P)$.
\end{lemma}
The proof of Lemma \ref{lem-approx1} is analogous to the proof of the $L^2(\mathbb P)$-approximation result for the nodal length of Random Spherical Harmonics in \cite[Appendix B]{MRW} and hence omitted. In order to state the next result we need to introduce two collections of coefficients
$\{\alpha_{2n,2m} : n,m\geq 1\}$ and $\{\beta_{l}(u) : l\geq 0\}$, that are related to the Hermite expansion of the norm $\| \cdot
\|$ in $\R^2$ and the (formal) Hermite expansion of the Dirac mass $ \delta_u(\cdot)$ respectively, cf. (\ref{ir1}).
These are given by
\begin{equation}\label{e:beta}
\beta_{l}(u):= H_{l}(u)\phi(u),
\end{equation}
where $H_{l}$ still denotes the $l$-th Hermite polynomial (\ref{herm}), and
\begin{equation}\label{e:alpha}
\alpha_{2n,2m}:=\sqrt{\frac{\pi}{2}}\frac{(2n)!(2m)!}{n!
m!}\frac{1}{2^{n+m}} p_{n+m}\left (\frac14 \right),
\end{equation}
where for $N\in \mathbb N$ and $x\in \R$
\begin{equation*}
\displaylines{ p_{N}(x) :=\sum_{j=0}^{N}(-1)^{j}\cdot(-1)^{N}{N
\choose j}\ \ \frac{(2j+1)!}{(j!)^2} x^j, }
\end{equation*}
the ratio $\frac{(2j+1)!}{(j!)^2}$ being the so-called swinging factorial
restricted to odd indices. The proof of the following lemma is analogous to the proof of Proposition 3.2 in \cite{MPRW2015} and hence omitted for the sake of brevity.

\begin{lemma}\label{lem-ce1}
For every $n\in S$ and $u\in \mathbb R$ the chaotic expansion of $\mathcal L_1(n;u)$ is
\begin{eqnarray*}
\mathcal L_1(n;u) &=& \frac12 \sqrt{\frac{E_n}{2}}\sum_{q=0}^{+\infty}\sum_{u=0}^{q}\sum_{k=0}^{u}
\frac{\alpha _{2k,2u-2k}\beta _{q-2u}(u)
}{(2k)!(2u-2k)!(q-2u)!}\times\\
\nonumber
&&\times \int_{\mathbb T}H_{q-2u}(f_n(x))
H_{2k}(f_{n,1}(x))H_{2u-2k}(f_{n,2}(x))\,dx,
\end{eqnarray*}
where the convergence of the series is in $L^2(\P)$, and $f_{n,\ell}$ denotes normalized first derivatives defined in (\ref{der2}).
\end{lemma}

\subsubsection{Euler-Poincar\'e characteristic}

The zero-th Lipschitz-Killing curvature has the following \emph{formal} representation
\begin{equation}\label{ir0}
\mathcal L_0(n;u) = \int_{\mathbb T} \text{\rm det}(\nabla^2 f_n(x)) \mathbf{1}_{\lbrace f_n(x)\ge u\rbrace} \delta_0(\nabla f_n(x))\,dx,
\end{equation}
where $\nabla^2 f_n$ is the Hessian matrix of $f_n$, and abusing notation $\delta_0$ denotes the Dirac mass in $(0,0)$.
For $\epsilon>0$, let us consider the $\epsilon$-approximating random variable
\begin{equation}\label{approx0}
\mathcal L_0^\epsilon(n;u) = \int_{\mathbb T}  \text{\rm det}(\nabla^2 f_n(x)) \mathbf{1}_{\lbrace f_n(x)\ge u\rbrace} \frac{1}{(2\epsilon)^2}\mathbf{1}_{[-\epsilon, \epsilon]^2}(\nabla f_n(x))\,dx.
\end{equation}
\begin{lemma}\label{ubEPC}
For every $n\in S$, $\epsilon>0$ and $u\in \mathbb R$
\begin{equation}\label{eq:ubEPC}
|\mathcal L^\epsilon_0(u;n)| \le 4E_n.
\end{equation}
\end{lemma}
The proof of Lemma \ref{ubEPC} is postponed to the Appendix \ref{app-approx}.
\begin{lemma}\label{approx-conv}
For $n\in S$ and $u\in \mathbb R$ it holds that, as $\epsilon \to 0$,
\begin{equation*}
\mathcal L_0^\epsilon(n;u) \to \mathcal L_0(n;u),
\end{equation*}
where the convergence is a.s. and in $L^2_{\mathcal A}(\mathbb P)$.
\end{lemma}
Equation (\ref{ir0}) is justified by Lemma \ref{approx-conv} whose proof is postponed to the Appendix \ref{app-approx}. In view of Lemma \ref{approx-conv}, by letting $\epsilon$ tend to zero in (\ref{eq:ubEPC}) we find
$$|\mathcal L_0(n;u) |\le 4E_n$$
 for every $n\in S$ e $u\in \mathbb R$, in particular $\mathcal L_0(n;u)$ belongs to $L^2_{\mathcal A}(\mathbb P)$.
Obviously once the a.s. convergence is proven, it suffices to apply Lemma \ref{ubEPC} to get the convergence in $L^2(\mathbb P)$.

The next result (whose proof will be given in \S \ref{sec-ce0}) concerns the chaotic expansion of $\mathcal L_0(n;u)$: we will not need explicit expressions for chaotic coefficients but those corresponding to the zero-th and second Wiener chaoses, see \S \ref{sub-proofs} and \S \ref{sec-2} respectively.
\begin{lemma}\label{lem-ce0}
For $n\in S$ and $u\in \mathbb R$, the chaotic expansion of $\mathcal L_0(n;u)$ is
 \begin{equation}\label{exp_L}
  \begin{split}
  \mathcal L_0(n;u)
    =& 2E_n\sum_{q=0}^{+\infty} \sum_{a+b+c+2d+2e=q} \frac{\eta^{(n)}_{a,b,c}(u)}{a!b!c!}\frac{\beta_{2d} \beta_{2e}}{(2d)!(2e)!} \int_{\mathbb T} H_{a}\left (\frac{\partial_{11}f_n(x)}{k_3} \right ) \\
    &\times H_b \left (\frac{\partial_{12}f_n(x)}{k_4}\right )
     H_c \left (\frac{\partial_{22}f_n(x)}{k_5} - \frac{k_2}{k_5 k_3} \partial_{11}f_n(x) \right )
    H_{2d}\left (\frac{\partial_1 f_n(x)}{k_1}\right )\\
    &\times  H_{2e}\left (\frac{\partial_2 f_n(x)}{k_1} \right )\,dx,
  \end{split}
  \end{equation}
 for some coefficients $\eta^{(n)}_{a,b,c}(u)\in \mathbb R, a,b,c\in \mathbb N$, where the series converges in $L^2(\mathbb P)$,
  \begin{equation}\label{e:beta0}
  \beta_q := \beta_q(0) = \phi(0) H_q(0)
  \end{equation}
  as defined in (\ref{e:beta}),
  and $k_1,\dots,k_5$ are defined in (\ref{matrix_K}).
\end{lemma}

\subsection{Proofs}\label{sub-proofs}

\begin{proof}[Proof of Lemma \ref{lem01}]
From Lemma \ref{lem-ce2} we have
$$
\text{\rm Proj}[\mathcal L_2(n;u) |0] = \frac{\gamma_0(u)}{0!} \int_{\mathbb T}H_0(f_n(x))\,dx = 1-\Phi(u),
$$
that coincides with $\mathbb E[\mathcal L_0(n;u)]$ in Lemma \ref{expval}. From Lemma \ref{lem-ce1}, (\ref{e:beta}) and (\ref{e:alpha})
\begin{eqnarray*}
\text{\rm Proj}[\mathcal L_1(n;u) |0] &=& \frac12 \sqrt{\frac{E_n}{2}}\frac{\alpha_{0,0}\beta_0(u)}{0!0!0!} \int_{\mathbb T}H_{0}(f_n(x))
H_{0}(f_{n,1}(x))H_{0}(f_{n,2}(x))\,dx\\
&=& \frac12 \sqrt{\frac{E_n}{2}}\sqrt{\frac{\pi}{2}} \phi(u)
\end{eqnarray*}
that is $\mathbb E[\mathcal L_1(n;u)]$ given in Lemma \ref{expval}. Let us now focus on $\mathcal L_0(n;u)$. Exploiting the proof of Lemma \ref{expval} in Appendix \ref{Appendixexp} we have that for every $n\in S$ and $u\in \mathbb R$
\begin{equation}\label{eta000}
\eta_{0,0,0}^{(n)}(u) = \frac14 u\phi(u),
\end{equation}
 hence
from Lemma \ref{lem-ce0},
\begin{eqnarray*}
\text{\rm Proj}[\mathcal L_0(n;u) |0] &=&  2 E_n\cdot  \eta^{(n)}_{0,0,0}(u)\beta^2_{0} \cr
&=&  \frac{1}{2\pi} u\phi(u) \cdot \frac{E_n}{2}
\end{eqnarray*}
that is equal to $\mathbb E[\mathcal L_0(n;u)]$ in Lemma \ref{expval}. From Lemma \ref{lem-ce2} and $n\in S$, $n\ge 1$, the first order chaotic component of $\mathcal L_2(n;u)$ is
\begin{equation}\label{proj1}
\begin{split}
\text{\rm Proj}[\mathcal L_2(n;u)|1] &= \frac{\gamma_1(u)}{1!}\int_{\mathbb T} H_1(f_n(x))\,dx=\phi(u) \int_{\mathbb T} f_n(x)\,dx\\
&=  \frac{\phi(u)}{\sqrt{\mathcal N_n}} \sum_{\lambda\in \Lambda_n}\underbrace{\int_{\mathbb T} e(\langle \lambda, x \rangle)\,dx}_{= \delta_\lambda^0} = 0.
\end{split}
\end{equation}
The proof of (\ref{ca1}) for $k=0,1$ is analogous to (\ref{proj1}), hence we omit the details.

\end{proof}

The proofs of Proposition \ref{secondchaosb} and Proposition \ref{prop-ho} are long and technical, hence postponed to \S \ref{sec-2c} and \S \ref{sec-ho} respectively.

\begin{proof}[Proof of Theorem \ref{mainterm} assuming  Proposition \ref{secondchaosb} and Proposition \ref{prop-ho}] From Lemma \ref{lem01} and (\ref{ce1}) we can write
\begin{equation}\label{e:app1}
\overline{\mathcal{L}}_{k}(n;u) =\sum_{q=2}^{\infty }\text{\rm Proj}[\mathcal{L}_{k}(n;u)|q],
\end{equation}
where the convergence of the (orthogonal) series is in $L^2(\mathbb P)$. Proposition \ref{secondchaosb} ensures that we can rewrite (\ref{e:app1}) as
\begin{equation*}
\overline{\mathcal{L}}_{k}(n;u) = \frac{c_k(u)}{\sqrt{\mathcal N_n/2}} \left (\sqrt{\frac{E_n}{2}} \right )^{2-k}   \frac{1}{\sqrt{\mathcal{N}_{n}/2} }\sum_{\lambda\in \Lambda^+_n }(|a_{\lambda }|^{2}-1) + \sum_{q=3}^{\infty }\text{\rm Proj}[\mathcal{L}_{k}(n;u)|q],
\end{equation*}
where $c_k(u)$ are given in (\ref{ck}). In order to conclude the proof it suffices to set
\begin{equation}\label{defR}
\mathcal R_k(n;u) := \sum_{q=3}^{\infty }\text{\rm Proj}[\mathcal{L}_{k}(n;u)|q]
\end{equation}
and recall Proposition \ref{prop-ho}.

\end{proof}

\begin{proof}[Proof of Corollary \ref{cor1}] Let us bear in mind Theorem \ref{mainterm} and (\ref{var-asymp}). For $n\in S$, $k=0,1,2$ and $u\notin \mathcal U_k$, by triangle inequality,
\begin{equation}\label{dt}
d_{W}\left( \widetilde{\mathcal L}_k(n;u),Z\right)
\leq d_{W}\left(  \widetilde{\mathcal L}_k(n;u),\frac{\mathrm{Proj}[\mathcal{L}_{k}(n;u)|2]}{\sqrt{\Var(\mathcal{L}_{k}(n;u))}}\right) + d_{W}\left( \frac{\mathrm{Proj}[\mathcal{L}_{k}(n;u)|2]}{\sqrt{\Var(\mathcal{L}_{k}(n;u))}},
Z\right).
\end{equation}
For the first term on the right-hand side of (\ref{dt}) it suffices to note that, by definition of Wasserstein distance and (\ref{ce1}),
\begin{equation}\label{ee1}
\begin{split}
d_{W}\left(  \widetilde{\mathcal L}_k(n;u),\frac{\mathrm{Proj}[\mathcal{L}_{k}(n;u) |2]}{\sqrt{\Var(\mathcal{L}_{k}(n;u))}}\right) &\le \sqrt{\mathbb E\left[ \left ( \frac{\mathcal R_k(n;u)}{\sqrt{\Var(\mathcal{L}_{k}(n;u))}}\right )^2\right ]} \\
& = O\left ( \sqrt{\frac{1}{\mathcal N_n}}\right ),
\end{split}
\end{equation}
where the last estimate follows from (\ref{err}) in Theorem \ref{mainterm} and (\ref{var-asymp}). The second term on the right hand side of (\ref{dt}) can be controlled by Berry-Esseen's bounds (see e.g. \cite{Ess42}), indeed the second order chaotic projection is a sum of i.i.d. random variables, see Proposition \ref{secondchaosb}, recall also (\ref{var-asymp}):
\begin{eqnarray}\label{ee2}
d_{W}\left( \frac{\mathrm{Proj}[\mathcal{L}_{k}(n;u)|2]}{\sqrt{\Var(\mathcal{L}_{k}(n;u))}},
Z\right) &\le& d_{W}\left( \frac{\mathrm{Proj}[\mathcal{L}_{k}(n;u)|2]}{\sqrt{\Var(\mathrm{Proj}[\mathcal{L}_{k}(n;u)|2])}}, Z\right)\cr
&&+ d_{W}\left( \frac{\mathrm{Proj}[\mathcal{L}_{k}(n;u)|2]}{\sqrt{\Var(\mathrm{Proj}[\mathcal{L}_{k}(n;u)|2])}},\frac{\mathrm{Proj}[\mathcal{L}_{k}(n;u)|2]}{\sqrt{\Var(\mathcal{L}_{k}(n;u))}}\right)\cr
&\le& d_{W}\left( \frac{1}{\sqrt{\mathcal N_n/2}}\sum_{\lambda\in \Lambda_n^+} (|a_\lambda|^2-1), Z\right)\\
&&+ \frac{\left | \sqrt{\Var\left (\mathcal{L}_{k}(n;u)\right) }-\sqrt{\Var\left ( \mathrm{Proj}[\mathcal{L}_{k}(n;u)|2]\right) }\right |}{\sqrt{\Var\left (\mathcal{L}_{k}(n;u)\right) }} \cr
&=& O\left ( \frac{1}{\sqrt{\mathcal N_n}}\right )\nonumber.
\end{eqnarray}
Plugging (\ref{ee1}) and (\ref{ee2}) into (\ref{dt}) we conclude the proof.

\end{proof}

\begin{proof}[Proof of Corollary \ref{cor2}] The proof is analogous to the proof of Theorem 1.7 in \cite{MRT20}. Applying standard Large Deviation results \cite{DZ98} for sums of i.i.d. random variables we have that under (\ref{cond1}) the sequence
$$
\lbrace \widetilde{\text{\rm Proj}}[\mathcal L_k(n;u)|2]/s_{n;u}^{(k)}\rbrace_{n\in S}
$$
satisfies a Moderate Deviation principle as $\mathcal N_n\to +\infty$ with speed $(s_{n;u}^{(k)})^2$ and rate function $\mathcal I$,
 where
 $$
 \widetilde{\text{\rm Proj}}[\mathcal L_k(n;u)|2] := \frac{\text{\rm Proj}[\mathcal L_k(n;u)|2]}{\sqrt{\Var(\text{\rm Proj}[\mathcal L_k(n;u)|2])}}.
  $$
Moreover, for every $\delta>0$, under (\ref{cond1}),
$$
\limsup_{\mathcal N_n\to +\infty} \frac{1}{(s_{n;u}^{(k)})^2} \log \mathbb P\left (\left | \widetilde{\mathcal L}_k(n;u)/ s_{n;u}^{(k)} -  \widetilde{\text{\rm Proj}}[\mathcal L_k(n;u)|2] /s_{n;u}^{(k)} \right |>\delta \right ) = -\infty,
$$
i.e. the two sequence of random variables $\lbrace \widetilde{\text{\rm Proj}}[\mathcal L_k(n;u)|2]/s_{n;u}^{(k)}\rbrace_{n\in S}$ and $\lbrace \widetilde{\mathcal L}_k(n;u)/s_{n;u}^{(k)}\rbrace_{n\in S}$
are exponentially equivalent \cite[Definition 4.2.10]{DZ98} along subsequences of energy levels such that $\mathcal N_n\to +\infty$. Theorem 4.2.13 in \cite{DZ98} then ensures that $\lbrace \widetilde{\mathcal L}_k(n;u)/s_{n;u}^{(k)}\rbrace_{n\in S}$ satisfies a Moderate Deviation principle with the same speed and rate function as the sequence $\lbrace \widetilde{\text{\rm Proj}}[\mathcal L_k(n;u)|2]/s_{n;u}^{(k)}\rbrace_{n\in S}$.

\end{proof}

\begin{proof}[Proof of Corollary \ref{cor3}]
From Theorem \ref{mainterm} and (\ref{defR}) we can write
\begin{eqnarray}\label{aa1}
&&\Corr\left( \mathcal{L}_{k_{1}}(n;u_1),\mathcal{L}_{k_{2}}(n;u_2)\right) = \frac{\Cov\left( \mathcal{L}_{k_{1}}(n;u_1),\mathcal{L}_{k_{2}}(n;u_2)\right)}{\sqrt{\Var(\mathcal{L}_{k_{1}}(n;u_1)) \Var(\mathcal{L}_{k_{2}}(n;u_2))}}\cr
&& = \frac{\Cov\left(\text{\rm Proj}[\mathcal L_{k_1}(n;u_1)|2] + \mathcal R_{k_1}(n;u_1),\text{\rm Proj}[\mathcal L_{k_2}(n;u_2)|2] + \mathcal R_{k_2}(n;u_2)\right)}{\sqrt{\Var(\mathcal{L}_{k_{1}}(n;u_1)) \Var(\mathcal{L}_{k_{2}}(n;u_2))}}\cr
&&= \frac{\Cov\left(\text{\rm Proj}[\mathcal L_{k_1}(n;u_1)|2]),\text{\rm Proj}[\mathcal L_{k_2}(n;u_2)|2] \right)}{\sqrt{\Var(\mathcal{L}_{k_{1}}(n;u_1)) \Var(\mathcal{L}_{k_{2}}(n;u_2))}}\cr
&& +  \frac{\Cov\left( \mathcal R_{k_1}(n;u_1), \mathcal R_{k_2}(n;u_2)\right)}{\sqrt{\Var(\mathcal{L}_{k_{1}}(n;u_1)) \Var(\mathcal{L}_{k_{2}}(n;u_2))}}\cr
&&= \frac{\Cov\left(\text{\rm Proj}[\mathcal L_{k_1}(n;u_1)|2]),\text{\rm Proj}[\mathcal L_{k_2}(n;u_2)|2] \right)}{\sqrt{\Var(\text{\rm Proj}[\mathcal L_{k_1}(n;u_1)|2]) \Var(\text{\rm Proj}[\mathcal L_{k_2}(n;u_2)|2] )}} + O\left ( \frac{1}{\mathcal N_n}\right )\cr
&&  +  \frac{\Cov\left( \mathcal R_{k_1}(n;u_1), \mathcal R_{k_2}(n;u_2)\right)}{\sqrt{\Var(\mathcal{L}_{k_{1}}(n;u_1)) \Var(\mathcal{L}_{k_{2}}(n;u_2))}},
\end{eqnarray}
where for the last equality we used (\ref{var-asymp}). From Proposition \ref{secondchaosb} we get
\begin{equation}\label{aa2}
\frac{\Cov\left(\text{\rm Proj}[\mathcal L_{k_1}(n;u_1)|2]),\text{\rm Proj}[\mathcal L_{k_2}(n;u_2)|2] \right)}{\sqrt{\Var(\text{\rm Proj}[\mathcal L_{k_1}(n;u_1)|2]) \Var(\text{\rm Proj}[\mathcal L_{k_2}(n;u_2)|2] )}} = 1
\end{equation}
and still from Theorem \ref{mainterm} and (\ref{var-asymp}), and Cauchy-Schwatz inequality, we get
\begin{equation}\label{aa3}
\begin{split}
\frac{\left | \Cov\left( \mathcal R_{k_1}(n;u_1), \mathcal R_{k_2}(n;u_2)\right) \right |}{\sqrt{\Var(\mathcal{L}_{k_{1}}(n;u_1)) \Var(\mathcal{L}_{k_{2}}(n;u_2))}}\le& \frac{\sqrt{\mathbb E[\mathcal R_{k_1}(n;u_1)^2]\mathbb E[\mathcal R_{k_2}(n;u_2)^2]}}{\sqrt{\Var(\mathcal{L}_{k_{1}}(n;u_1)) \Var(\mathcal{L}_{k_{2}}(n;u_2))}}\\
=& O\left ( \frac{1}{\mathcal N_n} \right ).
\end{split}
\end{equation}
Plugging (\ref{aa2}) and (\ref{aa3}) into (\ref{aa1}) we conclude the proof.

\end{proof}

\section{EPC: chaotic expansion}\label{sec-ce0}


\subsection{Cholesky decomposition}\label{sec-cho}

Recall the definition of the $\epsilon$-approximating random variable (\ref{approx0}) and Lemma \ref{approx-conv}. In order to prove Lemma \ref{lem-ce0} we first derive the chaotic expansion of $\mathcal L^\epsilon_0(n;u)$ and then let $\epsilon$ go to zero.
The integrand function
\begin{equation}\label{F}
F_n^\epsilon(x):= \text{\rm det}(\nabla^2 f_n(x)) \mathbf{1}_{\lbrace \Delta f_n(x)\le - E_nu\rbrace} \frac{1}{(2\epsilon)^2}\mathbf{1}_{[-\epsilon, \epsilon]^2}(\nabla f_n(x)),\qquad x\in \mathbb T,
 \end{equation}
 defining the $\epsilon$-approximating random variable (\ref{approx0}) is a functional of $\partial _{1}f_{n}$, $\partial _{2}f_{n}$, $\partial _{11}f_{n}$, $\partial
_{12}f_{n}$, $\partial _{22}f_{n}$ which are not point-wise independent random fields. For the sake of simplicity we first express $F^\epsilon(x)$ in terms of independent random variables.
Let us write $\sigma _{n}=\sigma _{n}(x)$ for the $5\times 5$ covariance
matrix (see \S \ref{cov=}) of the Gaussian random vector
\begin{equation*}
(\partial _{1}f_{n}(x),\partial _{2}f_{n}(x),\partial _{11}f_{n}(x),\partial
_{12}f_{n}(x),\partial _{22}f_{n}(x)).
\end{equation*}%
We write it in the partitioned form
\begin{equation*}
\sigma _{n} = \sigma _{n}(x)_{5\times 5} = \left(
\begin{array}{cc}
a_{n} & b_{n} \\
b_{n}^{t} & c_{n}%
\end{array}%
\right) ,
\end{equation*}%
where the superscript $t$ denotes transposition, and (see Appendix \ref%
{cov=})
\begin{equation*}
a_{n}=a_{n}(x)=\frac{4\pi ^{2}}{\mathcal{N}_{n}}\sum_{\lambda }\lambda
\lambda ^{t}=2\pi ^{2}n\;I_{2},
\end{equation*}
\begin{equation*}
b_{n}=b_{n}(x)=\left(
\begin{array}{ccc}
0 & 0 & 0 \\
0 & 0 & 0%
\end{array}%
\right) ,
\end{equation*}%
since for stationarity for random fields second derivatives and first
derivatives at every fixed point are uncorrelated [see \cite{adlertaylor},
page 114], and (recall (\ref{muquattro}))%
\begin{equation*}
c_{n}=c_{n}(x)=2\pi ^{4}n^{2}\left(
\begin{array}{ccc}
3+\hat{\mu}_{n}(4) & 0 & 1-\hat{\mu}_{n}(4) \\
0 & 1-\hat{\mu}_{n}(4) & 0 \\
1-\hat{\mu}_{n}(4) & 0 & 3+\hat{\mu}_{n}(4)%
\end{array}%
\right).
\end{equation*}
Via Cholesky decomposition we can write the Hermitian positive-definite
matrix $\sigma _{n}$ in the form $\sigma _{n}=K_{n}K_{n}^{t}$, where $K_{n}$
is a lower triangular matrix with real and positive diagonal entries, and $%
K_{n}^{t}$ denotes the conjugate transpose of $K_{n}$. By an explicit
computation, it is possible to show that the Cholesky decomposition of $%
\sigma _{n}$ takes the form $\sigma _{n}=K_{n}K_{n}^{t}$, where

\begin{equation*}
K_{n}=\left(
\begin{array}{ccccc}
\sqrt{2n}\,\pi  & 0 & 0 & 0 & 0 \\
0 & \sqrt{2n}\,\pi  & 0 & 0 & 0 \\
0 & 0 & \sqrt{2}\pi ^{2}n\sqrt{3+\hat{\mu}_{n}(4)} & 0 & 0 \\
0 & 0 & 0 & \sqrt{2}\pi ^{2}n\sqrt{1-\hat{\mu}_{n}(4)} & 0 \\
0 & 0 & \sqrt{2}\pi ^{2}n\frac{1-\hat{\mu}_{n}(4)}{\sqrt{3+\hat{\mu}_{n}(4)}}
& 0 & 4\pi ^{2}n\frac{\sqrt{1+\hat{\mu}_{n}(4)}}{\sqrt{3+\hat{\mu}_{n}(4)}}%
\end{array}%
\right)
\end{equation*}%
\begin{equation*}
=\left(
\begin{array}{ccccc}
\frac{\sqrt{E_{n}}}{\sqrt{2}} & 0 & 0 & 0 & 0 \\
0 & \frac{\sqrt{E_{n}}}{\sqrt{2}} & 0 & 0 & 0 \\
0 & 0 & \frac{E_{n}}{2\sqrt{2}}\sqrt{3+\hat{\mu}_{n}(4)} & 0 & 0 \\
0 & 0 & 0 & \frac{E_{n}}{2\sqrt{2}}\sqrt{1-\hat{\mu}_{n}(4)} & 0 \\
0 & 0 & \frac{E_{n}}{2\sqrt{2}}\frac{1-\hat{\mu}_{n}(4)}{\sqrt{3+\hat{\mu}%
_{n}(4)}} & 0 & E_{n}\frac{\sqrt{1+\hat{\mu}_{n}(4)}}{\sqrt{3+\hat{\mu}%
_{n}(4)}}%
\end{array}%
\right)
\end{equation*}%
\begin{equation}\label{matrix_K}
=:\left(
\begin{array}{ccccc}
k_{1} & 0 & 0 & 0 & 0 \\
0 & k_{1} & 0 & 0 & 0 \\
0 & 0 & k_{3} & 0 & 0 \\
0 & 0 & 0 & k_{4} & 0 \\
0 & 0 & k_{2} & 0 & k_{5}%
\end{array}%
\right).
\end{equation}%
We can hence introduce a $5$-dimensional standard Gaussian vector
\begin{equation*}
Y(x)=(Y_{1}(x),Y_{2}(x),Y_{3}(x),Y_{4}(x),Y_{5}(x))
\end{equation*}%
with independent components such that
\begin{align*}
& (\partial _{1}f_{n}(x),\partial _{2}f_{n}(x),\partial
_{11}f_{n}(x),\partial _{12}f_{n}(x),\partial _{22}f_{n}(x))=K_{n}Y(x) \\
&
=(k_{1}Y_{1}(x),k_{1}Y_{2}(x),k_{3}Y_{3}(x),k_{4}Y_{4}(x),k_{5}Y_{5}(x)+k_{2}Y_{3}(x)).
\end{align*}
Hence the expression (\ref{F}) can be rewritten as
\begin{eqnarray}\label{ciao1}
F_n^\epsilon(x) &= &[k_3 Y_3(x) (k_5 Y_5(x)+k_2 Y_3(x))-(k_4 Y_4(x))^2] \; \mathbf{1}_{\{
\frac{ k_3}{E_n} Y_3(x) + \frac{k_5}{E_n} Y_5(x) + \frac{k_2}{E_n} Y_3(x)
\le -u \}}\cr
&&\times  \frac{1}{(2\epsilon)^2}\mathbf{1}_{[-\epsilon, \epsilon]^2}(k_1 Y_1(x), k_1 Y_2(x))\cr
&=&  [k_3 Y_3(x) (k_5 Y_5(x)+k_2 Y_3(x))-(k_4 Y_4(x))^2] \; \mathbf{1}_{\{
\frac{ k_3}{E_n} Y_3(x) + \frac{k_5}{E_n} Y_5(x) + \frac{k_2}{E_n} Y_3(x)
\le -u \}}\cr
&&\times  \frac{1}{k_1^2}\frac{1}{\left (\frac{2\epsilon}{k_1} \right )^2}\mathbf{1}_{\left [-\frac{\epsilon}{k_1}, \frac{\epsilon}{k_1} \right ]^2}(Y_1(x), Y_2(x)).
\end{eqnarray}
 Let us now set
 $$\widetilde k_i := \frac{k_i}{E_n},\qquad i=2,3,4,5,$$
and note that
$\max_{i=2,\dots,5} \widetilde k_i = O(1),
$ where the constant involved in the $O$-notation is absolute. From (\ref{ciao1}) we can rewrite (\ref{approx0}) as
\begin{eqnarray}\label{exp_approx2}
\mathcal L_0^\epsilon (n;u) &= &\int_{\mathbb T} F^\epsilon_n(x)\,dx\cr
&= & \frac{2E_n}{\left (\frac{2\epsilon}{k_1} \right )^2} \int_{\mathbb T^2}  \left (\widetilde k_3 Y_3(x) (\widetilde k_5 Y_5(x) + \widetilde k_2 Y_3(x)) - \widetilde k_4^2 Y_4(x)^2   \right )  \cr
&& \times \mathbf{1}_{\lbrace (\widetilde k_3 + \widetilde k_2) Y_3(x) + \widetilde k_5 Y_5(x) \le - u \rbrace}\mathbf{1}_{\left [-\frac{\epsilon}{k_1}, \frac{\epsilon}{k_1} \right ]^2}\left ( Y_1(x), Y_2(x)\right )\,dx.
\end{eqnarray}

\subsection{Proof of Lemma \ref{lem-ce0}}

We will need the following preliminary results.

\begin{lemma}\label{lemc2}
For every $\epsilon>0$, $n\in S$ and $x\in \mathbb T$, the chaotic expansion of
$$
\delta^{\epsilon/k_1}(Y_1(x), Y_2(x)):= \frac{1}{\left (\frac{2\epsilon}{k_1} \right )^2}\mathbf{1}_{\left [-\frac{\epsilon}{k_1}, \frac{\epsilon}{k_1} \right ]^2}(Y_1(x), Y_2(x))
$$
is the following series
\begin{eqnarray}
\delta^{\epsilon/k_1}(Y_1(x), Y_2(x)) = \sum_{q=0}^{+\infty} \sum_{q'=0}^q \frac{\beta^{\epsilon/k_1}_{2q'}}{(2q')!}\frac{\beta^{\epsilon/k_1}_{2q-2q'}}{(2q-2q')!} H_{2q'}(Y_1(x)) H_{2q-2q'}(Y_2(x)),
\end{eqnarray}
 where the convergence is in the $L^2(\mathbb P)$-sense, and for $q\in \mathbb N_{\ge 1}$,
 \begin{equation}\label{beta-eps}
\beta^{\epsilon/k_1}_0=  \frac{1}{\frac{2\epsilon}{k_1}}\int_{-\epsilon/k_1}^{\epsilon/k_1}\phi(t)\,dt,\qquad  \beta^{\epsilon/k_1}_{2q} := \frac{1}{\frac{2\epsilon}{k_1}} \int_{-\epsilon/k_1}^{\epsilon/k_1} H_{2q}(t) \phi(t)\,dt.
 \end{equation}
\end{lemma}

\begin{lemma}\label{lemc0}
For every $n\in S$, $x\in \mathbb T$ and $u\in \mathbb R$, the chaotic expansion of
$$
p_n(Y_3(x),Y_4(x),Y_5(x))\mathbf{1}_{\lbrace (\widetilde k_3 + \widetilde k_2) Y_3(x) + \widetilde k_5 Y_5(x) \le - u \rbrace},
$$
where $p_n(Y_3(x),Y_4(x),Y_5(x)) := \widetilde k_3 Y_3(x) (\widetilde k_5 Y_5(x) + \widetilde k_2 Y_3(x)) - \widetilde k_4^2 Y_4(x)^2$,
is
\begin{eqnarray}\label{1}
&&p_n(Y_3(x),Y_4(x),Y_5(x)) \mathbf{1}_{\lbrace (\widetilde k_3 + \widetilde k_2) Y_3(x) + \widetilde k_5 Y_5(x) \le - u \rbrace}\cr
&&= \sum_{q=0}^{+\infty}\sum_{a+b+c=q} \frac{\eta^{(n)}_{a,b,c}(u)}{a!b!c!}  H_{a}(Y_3(x))H_{a}(Y_4(x)) H_{c}(Y_5(x)),
\end{eqnarray}
where the series converges in $L^2(\mathbb P)$.
\end{lemma}
We omit the proofs of Lemma \ref{lemc2} and Lemma \ref{lemc0}. We are now in a position to establish Lemma \ref{lem-ce0}.

\begin{proof}[Proof of Lemma \ref{lem-ce0}] Let us first find the chaotic decomposition of $F^\epsilon_n(x)$ in (\ref{F}).
  Lemma \ref{lemc2} together with Lemma \ref{lemc0} gives
  \begin{eqnarray*}
  F^\epsilon_n(x) &=& 2E_n\sum_{q=0}^{+\infty} \sum_{a+b+c+2d+2e=q} \frac{\eta^{(n)}_{a,b,c}(u)}{a! b! c!} \frac{\beta_{2d}^{\epsilon/k_1}}{(2d)!}\frac{\beta_{2e}^{\epsilon/k_1}}{(2e)!}\cr
  &&\times  H_{2d}(Y_1(x)) H_{2e}(Y_2(x))H_{a}(Y_3(x)) H_b(Y_4(x)) H_c(Y_5(x))
    \end{eqnarray*}
  entailing that the chaotic expansion of $\mathcal L_0(n;u)$ in (\ref{exp_approx2}) is
   \begin{eqnarray}\label{su3}
  \mathcal L_0^\epsilon(n;u) = && 2E_n \sum_{q=0}^{+\infty} \sum_{a+b+c+2d+2e=q} \frac{\eta^{(n)}_{a,b,c}(u)}{a! b! c!} \frac{\beta_{2d}^{\epsilon/k_1}}{(2d)!}\frac{\beta_{2e}^{\epsilon/k_1}}{(2e)!}\cr
  &&\times  \int_{\mathbb T} H_{2d}(Y_1(x)) H_{2e}(Y_2(x))H_{a}(Y_3(x)) H_b(Y_4(x)) H_c(Y_5(x))\,dx.
  \end{eqnarray}
  Now note that, as $\epsilon\to 0$, for every $q\in \mathbb N$
  \begin{equation*}
  \beta^{\epsilon/k_1}_q \to \beta_q
  \end{equation*}
  as defined in (\ref{e:beta0}). Hence, bearing in mind the Cholesky decomposition in \S \ref{sec-cho},  Lemma \ref{approx-conv} together with (\ref{su3}) allows to get  (\ref{exp_L}) thus concluding the proof.

  \end{proof}

\section{Second chaotic components}\label{sec-2c}

\subsection{EPC: preliminary results}

We compute now the projection coefficients of $\mathcal L_0(n;u)$ on second Wiener chaos.
From (\ref{exp_L}) in Lemma \ref{lem-ce0} we can write more compactly
  \begin{eqnarray}\label{exp_2}
 \text{\rm Proj}[ \mathcal L_0(n;u) |2]
    &=& 2E_n \sum_{a+b+c+2d+2e=2} \frac{\eta^{(n)}_{a,b,c}(u)}{a!b!c!}\frac{\beta_{2d} \beta_{2e}}{(2d)!(2e)!} \int_{\mathbb T} H_{a}\left (\frac{\partial_{11}f_n(x)}{k_3} \right ) \cr
    &&\times H_b \left (\frac{\partial_{12}f_n(x)}{k_4}\right )
     H_c \left ( \frac{\partial_{22}f_n(x)}{k_5} - \frac{k_2}{k_5 k_3} \partial_{11}f_n(x)\right )
    H_{2d}\left (\frac{\partial_1 f_n(x)}{k_1}\right )\cr
    &&\times  H_{2e}\left (\frac{\partial_2 f_n(x)}{k_1} \right )\,dx\\
    && =\sum_{i=1}^{5}%
\sum_{j=1}^{i-1}h_{ij}(u;n)\int_{\mathbb{T}}Y_{i}(x)Y_{j}(x)dx +\frac{1}{2}%
\sum_{i=1}^{5}h_{i}(u;n)\int_{\mathbb{T}}H_{2}(Y_{i}(x))dx,\nonumber
  \end{eqnarray}
where  for $i,j=1,\dots 5$, $i\neq j$
\begin{equation*}
h_{ij}(u;n ):=\lim_{\varepsilon \to 0} \mathbb{E}\left[ k _{3}Y_{3}
(k_{5}Y_{5}+k_{2}Y_{3})-(k_{4}Y_{4})^{2}]\;1\hspace{-0.27em}\mbox{\rm l}%
_{\left\{ \frac{k_{2}+k_{3}}{E }Y_{3}+\frac{k_{5}}{E }Y_{5}\leq -u\right\}
}\;\delta_{\varepsilon} (k_{1}Y_{1}, k_{1}Y_{2})Y_{i}Y_{j}\right];
\end{equation*}%
on the other hand, for $i=1,\dots 5$,
\begin{equation*}
h_{i}(u;n ):=\lim_{\varepsilon \to 0} \mathbb{E}\left[ k_{3}Y_{3}
(k_{5}Y_{5}+k_{2}Y_{3})-(k_{4}Y_{4})^{2}]\;1\hspace{-0.27em}\mbox{\rm l}%
_{\left\{ \frac{k_{2}+k_{3}}{E}Y_{3}+\frac{k_{5}}{E }Y_{5}\leq -u\right\}
}\;\delta_{\varepsilon} (k_{1}Y_{1}, k_{1}Y_{2})H_{2}(Y_{i})\right].
\end{equation*}
The following proposition provides analytic expressions for the coefficients
$h_{ij}$ and $h_{i}$.
\begin{proposition} It holds that
\label{proj} $h_{ij}(u;n )=0$ for all $(i,j) \ne (3,5)$ and
\begin{eqnarray*}
h_{35}(u; n )&=& \frac{E_n}{2 \sqrt 2 \pi} \sqrt{1+\hat{\mu}_n(4)} \frac{ u
\phi (u)(1+u^2) + (3+\hat{\mu}_n(4)) \Phi (-u) }{3+\hat{\mu}_n(4)}.
\end{eqnarray*}
Moreover
\begin{eqnarray*}
h_{1}(u;n )&=& h_{2}(u;n ) = - \frac{E_n}{4 \pi} u \,\phi(u), \\
h_{3}(u; n ) &=& \frac{E_n}{4 \pi} \left[ \frac{2 u (1+u^2) \phi(u)}{3+\hat{%
\mu}_n(4) } + \Phi(-u)(1 - \hat{\mu}_n(4)) \right], \\
h_{4}(u;n ) &=& - \frac{E_n}{4 \pi} (1-\hat{\mu}_n(4)) \Phi (-u), \\
h_{5}(u;n) &=& \frac{E_n}{4 \pi} \frac{ u (1+u^2) (1+\hat{\mu}_n(4))\phi (u)%
} {3+\hat{\mu}_n(4)}.
\end{eqnarray*}
\end{proposition}
The proof of Proposition \ref{proj} is postponed to the Appendix \ref{proofproj}. From equation (\ref{exp_2}) and Proposition \ref{proj} it is then immediate to obtain the
following expression:
\begin{equation}\label{sum1}
\text{\rm Proj}[\mathcal{L}_{0}(n;u)|2]=h_{35}(u;n)A_{35}(n)+\frac{1}{2}\sum_{i=1}^{5}h_{i}(u;n)B_{i}(n);
\end{equation}%
where
\begin{equation}\label{abab}
A_{ij}(n)=\int_{\mathbb{T}}Y_{i}(x)Y_{j}(x)dx,\hspace{1cm}B_{i}(n)=\int_{%
\mathbb{T}}H_{2}(Y_{i}(x))dx.
\end{equation}%
Our next step is then to investigate the behaviour of the integrals of
stochastic processes in (\ref{abab}).
\begin{proposition}
\label{AB} We have that
\begin{eqnarray*}
A_{35}(n) &=& \frac{1}{\mathcal{N}_{n}}\sum_{\lambda }|a_{\lambda }|^{2}\Big[
\frac{2\sqrt{2}}{n\sqrt{1+\hat{\mu}_{n}(4)}}\frac{5-\hat{\mu}_{n}(4)}{{3+%
\hat{\mu}_{n}(4)}}\lambda _{2}^{2}-\frac{1-\hat{\mu}_{n}(4)}{{3+\hat{\mu}%
_{n}(4)}}\frac{2\sqrt{2}}{\sqrt{1+\hat{\mu}_{n}(4)}}\cr
&&-\frac{\sqrt{2}}{n^{2}%
\sqrt{1+\hat{\mu}_{n}(4)}}\frac{8}{{3+\hat{\mu}_{n}(4)}}\lambda _{2}^{4}%
\Big] ,
\end{eqnarray*}%
and%
\begin{equation*}
B_{1}(n)=\frac{1}{\mathcal{N}_{n}}\sum_{\lambda }|a_{\lambda }|^{2}\frac{2}{n%
}\lambda _{1}^{2}-1,B_{2}(n)=\frac{1}{\mathcal{N}_{n}}\sum_{\lambda
}|a_{\lambda }|^{2}\frac{2}{n}\lambda _{2}^{2}-1,
\end{equation*}%
\begin{equation*}
B_{3}(n)=\frac{1}{\mathcal{N}_{n}}\sum_{\lambda }|a_{\lambda }|^{2}\left[
\frac{8}{3+\hat{\mu}_{n}(4)}+\frac{8}{n^{2}(3+\hat{\mu}_{n}(4))}\lambda
_{2}^{4}-\frac{16}{n(3+\hat{\mu}_{n}(4))}{\lambda _{2}^{2}}\right] -1,
\end{equation*}%
\begin{equation*}
B_{4}(n)=\frac{1}{\mathcal{N}_{n}}\sum_{\lambda }|a_{\lambda }|^{2}\frac{8}{%
n^{2}(1-\hat{\mu}_{n}(4))}\lambda _{1}^{2}\lambda _{2}^{2}-1,
\end{equation*}
\begin{eqnarray*}
B_{5}(n) &=& \frac{1}{\mathcal{N}_{n}}\sum_{\lambda }|a_{\lambda }|^{2}\Big[%
\frac{16}{n^{2}(1+\hat{\mu}_{n}(4))(3+\hat{\mu}_{n}(4))}\lambda _{2}^{4}+%
\frac{(1-\hat{\mu}_{n}(4))^{2}}{(3+\hat{\mu}_{n}(4))(1+\hat{\mu}_{n}(4))}\cr
&&-%
\frac{1-\hat{\mu}_{n}(4)}{n(3+\hat{\mu}_{n}(4))(1+\hat{\mu}_{n}(4))}8{%
\lambda _{2}^{2}}\Big]-1.
\end{eqnarray*}
\end{proposition}
The proof of Proposition \ref{AB} is technical hence postponed to the Appendix \ref{proofAB}.

\subsection{Proof of Proposition \ref{secondchaosb}}\label{sec-2}

\begin{proof} From Lemma \ref{lem-ce2} and (\ref{herm}) for $q=2$ we have
\begin{eqnarray*}
\text{\rm Proj}[\mathcal L_2(n;u)|2] &=& \frac{\gamma_2(u)}{2!}\int_{\mathbb T} H_2(f_n(x))\,dx \cr
&=& \frac12 u \phi(u) \frac{1}{\mathcal N_n/2} \sum_{\lambda\in \Lambda^+_n} (|a_\lambda|^2-1)
\end{eqnarray*}
which is (\ref{secondchaosb}) for $k=2$. Proposition \ref{secondchaosb} for $k=1$ has been proved in the spherical case in \cite[\S7.3]{Ros15} via an application of Green's formula, the proof for Arithmetic Random Waves is analogous hence omitted for the sake of brevity. Let us now focus on the Euler-Poincar\'e characteristic: plugging the results of Proposition \ref{proj} and Proposition \ref{AB} into (\ref{sum1}) some straightforward computations give Proposition \ref{secondchaosb} for $k=0$.

\end{proof}

\section{Higher order chaotic components}\label{sec-ho}

In this section we investigate higher order chaotic components of $\mathcal L_k(n;u)$, $k=0,1,2$. Let us start studying the variance of the projection onto the third Wiener chaos.
\begin{lemma}\label{lem:3}
For $k=0,1,2$, as $\mathcal N_n\to +\infty$,
\begin{equation}\label{var3}
\text{\rm Var}\left ( \mathcal L_k(n;u )[3] \right ) = O\left ( \frac{E_n^{2-k}}{\mathcal N_n^2}  \right ),
\end{equation}
where the constant involved in the $'O'$-notation does not depend on $n$.
\end{lemma}
\begin{proof}
From Lemma \ref{lem-ce2} and properties of Hermite polynomials (see e.g. \cite[Proposition 2.2.1]{noupebook}) we have
\begin{eqnarray}\label{bo1}
\Var(\mathcal L_2(n;u)[3]) &=& \frac{\gamma_3(u)^2}{3!} \int_{\mathbb T} r_n(x)^3\,dx\cr
&=& \frac{\gamma_3(u)^2}{3!}\cdot \frac{1}{\mathcal N_n^3}\sum_{\lambda, \lambda_1, \lambda_2\in \Lambda_n} \int_{\mathbb T} e_{\lambda + \lambda_1 + \lambda_2}(x)\,dx \cr
&=&  \frac{\gamma_3(u)^2}{3!}\cdot \frac{|S_3(n)|}{\mathcal N_n^3},
\end{eqnarray}
where
\begin{equation}\label{S3}
S_3(n) := \lbrace (\lambda, \lambda_1,\lambda_2)\in \Lambda_n^3 : \lambda + \lambda_1 + \lambda_2 =0\rbrace
\end{equation}
is the length$-3$ spectral correlation set. Reasoning as in \cite[p.31]{KKW} it is immediate to check that
\begin{equation}\label{eq3n}
|S_3(n)| = O(\mathcal N_n),
\end{equation}
hence from (\ref{bo1}) we deduce that
\begin{eqnarray}
\Var(\mathcal L_2(n;u)[3]) = O\left ( \frac{1}{\mathcal N_n^2}\right )
\end{eqnarray}
which is (\ref{var3}) for $k=2$. From Lemma \ref{lem-ce1} we have
\begin{eqnarray}
\Var(\mathcal L_1(n;u )[3]) &=& \frac14 \frac{E_n}{2}\sum_{{\substack{a+2b+2c=3\\a'+2b'+2c'=3}}}
\frac{\beta _{a}(u)\alpha _{2b,2c}
}{a!(2b)!(2c)!}\frac{\beta _{a'}(u)\alpha _{2b',2c'}
}{(a')!(2b')!(2c')!}\cr
&&\times \int_{\mathbb T}\mathbb E\big [H_{a}(f_n(x))
H_{2b}(f_{n,1}(x))H_{2c}(f_{n,2}(x))\cr
&&\times H_{a'}(f_n(y))
H_{2b'}(f_{n,1}(y))H_{2c'}(f_{n,2}(y)) \big ]\,dxdy.
\end{eqnarray}
From Lemma \ref{lem-ce0} we have
\begin{eqnarray}\label{app_v3}
\text{\rm Var}(\mathcal L_0(n;u )[3])  &= & 4E_n^2 \sum_{{\substack{a+b+c+2d+2e=3\\a'+b'+c'+2d'+2e'=3}}} \frac{\eta_{a,b,c}(u)}{a!b!c!}\frac{\beta_{2d} \beta_{2e}}{(2d)!(2e)!} \frac{\eta_{a',b',c'}(u)}{a'!b'!c'!}\frac{\beta_{2d'} \beta_{2e'}}{(2d')!(2e')!}\\
&&\times \iint_{\mathbb T\times \mathbb T} \mathbb E \big [ H_{a}(Y_3(x)) H_b(Y_5(x)) H_c(Y_4(x))  H_{2d}(Y_1(x)) H_{2e}(Y_2(x))\\
&&\times H_{a'}(Y_3(y)) H_{b'}(Y_5(y)) H_{c'}(Y_4(y))  H_{2d'}(Y_1(y)) H_{2e'}(Y_2(y)) \big ]\,dxdy.
\end{eqnarray}
Consider the term for $a=a'=3$, $b=b'=c=c'=d=d'=e=e'=0$. We have
\begin{eqnarray}\label{ex3}
&&4E_n^2\frac{\eta_{3,0,0}(u)^2}{(3!)^2}\beta_{0}^4 \iint_{\mathbb T \times \mathbb T} \mathbb E [Y_3(x) Y_3(y)]^3\,dxdy\cr
& =& 4E_n^2\frac{\eta_{3,0,0}(u)^2}{(3!)^2}\beta_{0}^4 \iint_{\mathbb T \times \mathbb T}  \frac{1}{k_3^6}\mathbb E \left [\partial_{11} f_n(x) \partial_{11} f_n(y) \right ]^3\,dxdy\cr
&= & 4E_n^2\frac{\eta_{3,0,0}(u)^2}{(3!)^2}\beta_{0}^4 \iint_{\mathbb T \times \mathbb T}  \frac{1}{k_3^6}\left ( \frac{16\pi^4}{\mathcal N_n}\sum_{\lambda \in \Lambda_n} \lambda_1^4\, \text{\rm e}^{i 2\pi\langle \lambda, x-y\rangle } \right )^3\,dxdy\cr
&=& 4E_n^2\frac{\eta_{3,0,0}(u)^2}{(3!)^2}\beta_{0}^4 \iint_{\mathbb T \times \mathbb T}  \frac{1}{k_3^6} \frac{(16\pi^4)^3}{\mathcal N^3_n}\sum_{\lambda,\lambda',\lambda'' \in \Lambda_n} \lambda_1^4(\lambda'_1) ^4(\lambda''_1) ^4\, \text{\rm e}^{i 2\pi\langle \lambda + \lambda'+\lambda'', x-y\rangle } \,dxdy\cr
&=& 4E_n^2\frac{\eta_{3,0,0}(u)^2}{(3!)^2}\beta_{0}^4  \frac{1}{k_3^6} \frac{(16\pi^4)^3}{\mathcal N^3_n}\sum_{(\lambda,\lambda',\lambda'') \in S_3(n)} \lambda_1^4(\lambda'_1) ^4(\lambda''_1) ^4\cr
&=& 4E_n^2\frac{\eta_{3,0,0}(u)^2}{(3!)^2}\beta_{0}^4 \underbrace{ \left ( \frac{n}{k_3} \right )^6 }_{\le 1}\frac{(16\pi^4)^3}{\mathcal N^3_n}\sum_{(\lambda,\lambda',\lambda'') \in S_3(n)} \underbrace{\left ( \frac{\lambda_1}{\sqrt n}\right )^4\left (\frac{\lambda'_1}{\sqrt n} \right ) ^4 \left (\frac{\lambda''_1}{\sqrt n} \right ) ^4}_{\le 1}\cr
&\le& 4(16\pi^4)^3\frac{\eta_{3,0,0}(u)^2\beta_{0}^4}{(3!)^2} \, E_n^2\frac{|S_3(n)|}{\mathcal N^3_n}.
\end{eqnarray}
Plugging (\ref{eq3n}) into (\ref{ex3}) and repeating a similar argument for the other summands on the right hand side of (\ref{app_v3}) we get (\ref{var3}) for $k=0$. The proof for (\ref{var3}) for $k=1$ is similar. hence we omit the details.

\end{proof}

Lemma \ref{lem:3} ensures that $ \mathcal L_k(n;u )[3]$ is asymptotically negligible with respect to $\mathcal L_k(n;u )[2]$ for any $k\in \lbrace 0,1,2\rbrace$ (see Proposition \ref{secondchaosb}), as happens for the remaining chaotic projections. All the results to follow hold for every $n$ when $k=1,2$, and for $n\in S'$ for $k=0$.
For brevity's sake we avoid to repeat these conditions in the statements below.
\begin{lemma}\label{caos_sup}
As $\mathcal N_n\to +\infty$ under Condition \ref{condizione}
\begin{equation}
\text{\rm Var}\left ( \sum_{q=4}^{+\infty} \mathcal L_k(n;u )[q] \right ) = O\left ( \frac{E_n^{2-k}}{\mathcal N_n^2}  \right ),
\end{equation}
where the constant involved in the $'O'$-notation does not depend on $n$.
\end{lemma}
The proof of Proposition \ref{caos_sup} for $k=2$ is simple, the proof for $k=1$ can be treated analogously as the proof of Lemma 2 in \cite{PR18} hence we will omit both of them.

The proof of Proposition \ref{caos_sup} for $k=0$ is inspired by the proofs of Proposition 2.3 in \cite[\S 5]{dalmao}
and Lemma 2  in \cite{PR18}.
Let us first decompose $\mathbb T$ as a \emph{disjoint} union of squares $Q_k$, $k\in \mathbb Z^2$, each of side length $1/M$, where
\begin{equation}\label{def_M}
M = \lceil{d \sqrt{E_n}} \rceil, \qquad d\in \mathbb R_{>0} \text{ to be chosen later},
\end{equation}
obtained by translating along directions $k/M$, $k\in \mathbb Z^2$,  the square
\begin{equation}
Q_0 := [0, 1/M) \times [0, 1/M).
\end{equation}
In what follows we will often drop the dependence of $k$ from $Q_k$.

\subsection{Singular squares}

This part is inspired by \cite[\S6.1]{oravecz} and \cite[\S4.3]{Rudnick}. Let us fix $0<\epsilon \ll 1$ and choose $d$ in (\ref{def_M}) such that $d\ge c/\epsilon$; write 
$r_n(x-y)=r(x-y)=\mathbb E \left [f_n(x)f_n(y) \right ]$, $r_{n,1}(x-y)=r_1(x-y)=\mathbb E \left [\partial_{1} f_n(x)\partial_{1} f_n(y) \right ]$, $r_{n,2}(x-y)=r_2(x-y)=\mathbb E \left [\partial_{2} f_n(x)\partial_{2} f_n(y) \right ]$, and analogously for second-order derivatives.  A pair of points $(x,y)\in \mathbb T \times \mathbb T$ is said to be \emph{singular} if either $|r(x-y)|>\epsilon$ or $|r_1(x-y)| > \epsilon \, \sqrt{n}$ or $|r_2(x-y)| > \epsilon \, \sqrt{n}$ or $|r_{11}(x-y)| > \epsilon\, n$ or $|r_{12}(x-y)| > \epsilon\, n$ or $|r_{22}(x-y)| > \epsilon\, n$.
\begin{definition}
A pair of squares $(Q, Q')$ is said to be singular if there exists a singular pair of points $(x,y)\in Q\times Q'$.
\end{definition}
For instance, $(Q_0, Q_0)$ is a singular pair of squares. The following is Lemma 5.2 in \cite{dalmao}.
\begin{lemma}
Let $(Q,Q')$ be a singular pair of squares,  then for every $(x,y)\in (Q,Q')$ either $|r(x-y)|>\frac12 \epsilon$ or $|r_1(x-y)| > \frac12 \epsilon\, \sqrt{n}$ or $|r_2(x-y)| > \frac12 \epsilon\, \sqrt{n}$ or $|r_{11}(x-y)| > \frac12 \epsilon\, n$ or $|r_{12}(x-y)| > \frac12 \epsilon\, n$ or $|r_{22}(x-y)| > \frac12 \epsilon\, n$.
\end{lemma}
Let us now denote by $S_Q$ the union of all squares $Q'$ such that $(Q,Q')$ is a singular pair of squares. The number of such squares $Q'$ is $M^2\cdot  \text{area}(S_Q)$, indeed the area of each square is $1/M^2$. The following result is similar to Lemma 5.3 in \cite{dalmao} hence we omit the proof.
\begin{lemma}\label{lem_r4}
We have
\begin{equation}
\text{\rm area}(S_Q) \ll \int_{\mathbb T} r(x)^4\,dx.
\end{equation}
\end{lemma}
Recall that
\begin{equation}
\int_{\mathbb T} r(x)^4\,dx = \frac{|S_4(n)|}{\mathcal N_n^4},
\end{equation}
where $S_4(n):= \lbrace (\lambda, \lambda', \lambda'', \lambda''')\in \Lambda_n^4 : \lambda + \lambda' + \lambda'' + \lambda''' = 0\rbrace$ is the length$-4$ spectral correlation set, and \cite[Lemma 5.1]{MPRW2015} (see also \cite[p. 31]{KKW})
\begin{equation}\label{card_S4}
|S_4(n)| = 3 \mathcal N_n (\mathcal N_n-1).
\end{equation}
From Lemma \ref{lem_r4} and (\ref{card_S4}) we immediately have
\begin{equation}\label{area}
\text{\rm area}(S_Q) \ll \frac{1}{\mathcal N_n^2},
\end{equation}
thus the number of squares $Q'$ such that $(Q,Q')$ is singular is $\ll E_n/\mathcal N_n^2$.

\subsection{Variance and squares}

Let us denote by $\mathcal L_0(n;u,Q)$ the Euler-Poincar\'e characteristic restricted to $Q$.
Since the squares $Q$ are disjoint we can write
\begin{equation}\label{dubbio}
\mathcal L_0(n;u) = \sum_{Q} \mathcal L_0(n;u,Q)
\end{equation}
yielding
\begin{equation}\label{app1}
\text{proj}(\mathcal L_0(n;u) | C_{\ge 4}) = \sum_Q \text{\rm proj} (\mathcal L_0(n;u,Q) | C_{\ge 4}).
\end{equation}
From (\ref{app1}) we deduce
\begin{eqnarray}\label{app6}
\text{\rm Var}(\text{\rm proj}(\mathcal L_0(n;u) | C_{\ge 4}) &=& \sum_{Q,Q'} \text{\rm Cov} (\text{\rm proj}(\mathcal L_0(n;u,Q)|C_{\ge 4}), \text{\rm proj}(\mathcal L_0(n;u,Q')|C_{\ge 4}) ) \\
&=& \sum_{(Q,Q') \text{ \rm sing.}} \text{\rm Cov} (\text{\rm proj}(\mathcal L_0(n;u,Q)|C_{\ge 4}), \text{\rm proj}(\mathcal L_0(n;u,Q')|C_{\ge 4}) )\cr
&&+ \sum_{(Q,Q') \text{ \rm non-sing.}} \text{\rm Cov} (\text{\rm proj}(\mathcal L_0(n;u,Q)|C_{\ge 4}), \text{\rm proj}(\mathcal L_0(n;u,Q')|C_{\ge 4}) )\nonumber;
\end{eqnarray}
we are going to separately study the contribution of the singular part and the contribution of the non-singular part, i.e. the two summands on the right-hand-side of (\ref{app6}).

\subsection{The contribution of the singular part}

By Cauchy-Schwartz inequality and stationarity of the model we have
\begin{equation}\label{app2}
\begin{split}
&\left |  \sum_{(Q,Q') \text{ \rm sing.}} \text{\rm Cov} (\text{\rm proj}(\mathcal L_0(n;u,Q)|C_{\ge 4}), \text{\rm proj}(\mathcal L_0(n;u,Q')|C_{\ge 4}) ) \right |\\
&\le  \sum_{(Q,Q') \text{ \rm sing.}} \text{\rm Var} (\text{\rm proj}(\mathcal L_0(u;Q_0)|C_{\ge 4}))\\
&\ll M^2 \frac{E_n}{\mathcal N_n^2} \text{\rm Var} (\text{\rm proj}(\mathcal L_0(u;Q_0)|C_{\ge 4})) \ll \frac{E_n^2}{\mathcal N_n^2} \text{\rm Var} (\text{\rm proj}(\mathcal L_0(u;Q_0)|C_{\ge 4})),
 \end{split}
\end{equation}
where for the last two estimates we used (\ref{area}) and (\ref{def_M}) respectively. Now write
\begin{equation}\label{app3}
\begin{split}
\text{\rm Var} (\text{\rm proj}(\mathcal L_0(u;Q_0)|C_{\ge 4})) & \le  \mathbb E \left [\mathcal  L_0(u;Q_0)^2\right ]\\
&= \mathbb E \left [\mathcal L_0(u;Q_0) (\mathcal L_0(u;Q_0)-1) \right ] + \mathbb E \left [\mathcal L_0(u;Q_0)\right ].
\end{split}
\end{equation}
\begin{lemma}\label{final1}
For every $n\in S$ and $u\in \mathbb R$
\begin{equation*}
\mathbb E \left [\mathcal L_0(u;Q_0)\right ] = O(1),
\end{equation*}
where the constant involved in the $O$-notation is absolute.
\end{lemma}
The proof of Lemma \ref{final1} follows from the stationarity of the model and the fact that $\mathcal L_0(n;u)$ is bounded from above by $E_n$.
Lemma \ref{final1} together with (\ref{cond2}) entail that the right hand side of (\ref{app3}) is $O(1)$ hence
\begin{equation}
\text{\rm Var} (\text{\rm proj}(\mathcal L_0(u;Q_0)|C_{\ge 4})) = O(1),
\end{equation}
where the constant involved in the $O$-notation is absolute,
that together with (\ref{app2}) proves the following.
\begin{lemma}\label{lem_s}
As $\mathcal N_n\to +\infty$ under Condition \ref{condizione}
\begin{equation}\label{app4}
\left |  \sum_{(Q,Q') \text{ \rm sing.}} \text{\rm Cov} (\text{\rm proj}(\mathcal L_0(u;Q)|C_{\ge 4}), \text{\rm proj}(\mathcal L_0(u;Q')|C_{\ge 4}) ) \right | =O\left (  \frac{E_n^2}{\mathcal N_n^2}\right ),
\end{equation}
where the constant involved in the $O$-notation is absolute.
\end{lemma}

\subsection{The contribution of the non-singular part}

In this part we prove the following.
\begin{lemma}\label{lem_ns}
As $\mathcal N_n\to +\infty$
\begin{equation}\label{ap1}
\left |  \sum_{(Q,Q') \text{ \rm non-sing.}} \text{\rm Cov} (\text{\rm proj}(\mathcal L_0(u;Q)|C_{\ge 4}), \text{\rm proj}(\mathcal L_0(u;Q')|C_{\ge 4}) ) \right | =O\left (  \frac{E_n^2}{\mathcal N_n^2} \right ),
\end{equation}
where the constant involved in the $O$-notation is absolute.
\end{lemma}
\begin{proof}
As in the proof of Lemma 3.5 in \cite{dalmao} we can write
\begin{equation}\label{ap2}
\begin{split}
& \left |  \sum_{(Q,Q') \text{ \rm non-sing.}} \text{\rm Cov} (\text{\rm proj}(\mathcal L_0(u;Q)|C_{\ge 4}), \text{\rm proj}(\mathcal L_0(u;Q')|C_{\ge 4}) ) \right | \\
&\le 4E_n^2 \sum_{q\ge 4} \sum_{{\substack{a+b+c+2d+2e=q\\a'+b'+c'+2d'+2e'=q}}} \left | \frac{\eta^{(n)}_{a,b,c}(u)}{a!b!c!}\frac{\beta^{\epsilon/k_1}_{2d} \beta^{\epsilon/k_1}_{2e}}{(2d)!(2e)!} \frac{\eta^{(n)}_{a',b',c'}(u)}{a'!b'!c'!}\frac{\beta^{\epsilon/k_1}_{2d'} \beta^{\epsilon/k_1}_{2e'}}{(2d')!(2e')!} \right |\\
&\times \left | V(a,b,c,d,e,a',b',c',d',e') \right |,
\end{split}
\end{equation}
where $V(a,b,c,d,e,a',b',c',d',e')$ is the sum of no more than $q!$ terms of the type
\begin{equation}
v = \sum_{(Q,Q') \text{ \rm non-sing.}} \iint_{Q\times Q'} \prod_{u=1}^q R_{l_u,k_u}(x-y)\,dxdy,
\end{equation}
with $l_u, k_u\in \lbrace 0,1,2\rbrace$ and where for $l,k=0,1,2$ and $x,y\in \mathbb T$ we set
\begin{equation}
R_{l,k} (x-y) := \mathbb E[Y_{l}(x) Y_k(y)].
\end{equation}
For every integer $q\ge 4$
\begin{equation}\label{ap3}
\begin{split}
\left | v \right | &\le \sum_{(Q,Q') \text{ \rm non-sing.}} \iint_{Q\times Q'} | r_n(x-y) |^q\,dxdy\\
&\le \epsilon^{q-4}  \sum_{(Q,Q') \text{ \rm non-sing.}} \iint_{Q\times Q'} r_n(x-y)^4\,dxdy\\
& \le  \epsilon^{q-4} \int_{\mathbb T} r_n(x)^4\,dx.
\end{split}
\end{equation}
From (\ref{ap3}) we deduce that
\begin{equation}\label{ap4}
\begin{split}
|V(a,b,c,d,e,a',b',c',d',e')| & \le q! \frac{\int_{\mathbb T} r_n(x)^4\,dx}{\epsilon^4} \epsilon^q \\
& = q! \frac{\int_{\mathbb T} r_n(x)^4\,dx}{\epsilon^4}\, (\sqrt{\epsilon})^{a+b+c+2d+2e}(\sqrt{\epsilon})^{a'+b'+c'+2d'+2e'}.
\end{split}
\end{equation}
Plugging (\ref{ap4})  into (\ref{ap2}) we get
\begin{equation}\label{ap5}
\begin{split}
& \left |  \sum_{(Q,Q') \text{ \rm non-sing.}} \text{\rm Cov} (\text{\rm proj}(\mathcal L_0(u;Q)|C_{\ge 4}), \text{\rm proj}(\mathcal L_0(u;Q')|C_{\ge 4}) ) \right | \\
&\le 4E_n^2  \frac{\int_{\mathbb T} r_n(x)^4\,dx}{\epsilon^4}\sum_{q\ge 4} q! \sum_{{\substack{a+b+c+2d+2e=q\\a'+b'+c'+2d'+2e'=q}}}  \left | \frac{\eta^{(n)}_{a,b,c}(u)}{a!b!c!}\frac{\beta^{\epsilon/k_1}_{2d} \beta^{\epsilon/k_1}_{2e}}{(2d)!(2e)!} \frac{\eta^{(n)}_{a',b',c'}(u)}{a'!b'!c'!}\frac{\beta^{\epsilon/k_1}_{2d'} \beta^{\epsilon/k_1}_{2e'}}{(2d')!(2e')!} \right |\\
&\times\, (\sqrt{\epsilon})^{a+b+c+2d+2e}(\sqrt{\epsilon})^{a'+b'+c'+2d'+2e'};
\end{split}
\end{equation}
reasoning as in the proof of Lemma 3.5 in \cite{dalmao}, for $5\sqrt \epsilon <1$ we obtain
\begin{eqnarray}\label{Z}
&&Z:= \sum_{q\ge 4} q! \sum_{{\substack{a+b+c+2d+2e=q\\a'+b'+c'+2d'+2e'=q}}} \left | \frac{\eta^{(n)}_{a,b,c}(u)}{a!b!c!}\frac{\beta^{\epsilon/k_1}_{2d} \beta^{\epsilon/k_1}_{2e}}{(2d)!(2e)!} \frac{\eta^{(n)}_{a',b',c'}(u)}{a'!b'!c'!}\frac{\beta^{\epsilon/k_1}_{2d'} \beta^{\epsilon/k_1}_{2e'}}{(2d')!(2e')!} \right | \cr
&&\times (\sqrt{\epsilon})^{a+b+c+2d+2e}(\sqrt{\epsilon})^{a'+b'+c'+2d'+2e'} \cr
&&\le \sum_{a,b,c,d,e,a',b',c',d',e'} \left | \frac{\eta^{(n)}_{a,b,c}(u)}{a!b!c!}\frac{\beta^{\epsilon/k_1}_{d} \beta^{\epsilon/k_1}_{e}}{d! e!}  \right |^2 (a+b+c+d+e)! \cr
&& \times (\sqrt{\epsilon})^{a+b+c+d+e+a'+b'+c'+d'+e'} \\
&&\le \sum_{a,b,c,d,e,a',b',c',d',e'} \frac{\left | \eta^{(n)}_{a,b,c}(u)\beta^{\epsilon/k_1}_{d} \beta^{\epsilon/k_1}_{e}\right |^2}{a!b!c!d! e!}   5^{a+b+c+d+e}(\sqrt{\epsilon_1})^{a+b+c+d+e+a'+b'+c'+d'+e'}\nonumber.
\end{eqnarray}
Let us now prove that
\begin{equation}\label{ub}
(a,b,c,d,e)\mapsto \frac{\left | \eta^{(n)}_{a,b,c}(u)\beta^{\epsilon/k_1}_{d} \beta^{\epsilon/k_1}_{e}\right |^2}{a!b!c!d! e!}
\end{equation}
is uniformly bounded over $\epsilon$ and $n$. From (\ref{beta-eps}), recalling that there exists $C>0$ such that for every $q\in \mathbb N$ and $u\in \mathbb R$
$$
|H_q(u)|\phi(u) \le C \sqrt{q!},
$$
we have for every $\epsilon >0$, $n\in S$ and $d\in \mathbb N$,
\begin{equation}
 \frac{\left | \beta^{\epsilon/k_1}_{d} \right |^2}{d!} \le C^2.
\end{equation}
Moreover from Lemma \ref{lemc0} we have
\begin{eqnarray}
\sum_{q=1}^{+\infty} \sum_{a+b+c=q}  \frac{\left | \eta^{(n)}_{a,b,c}(u)\right |^2}{a!b!c!} &=& \mathbb E \left [p_n(Y_3(x), Y_4(x), Y_5(x)^2 \mathbf{1}_{(\widetilde k_3 +\widetilde k_2) Y_3(x) + \widetilde k_5 Y_5(x) \le -u} \right ]\cr
&\le&  \mathbb E \left [p_n(Y_3(x), Y_4(x), Y_5(x)^2  \right ] =O(1),\cr
\end{eqnarray}
where the constant involved in the $O$-notation is absolute.
Equation \ref{ub} together with (\ref{ap5}), (\ref{Z}) and (\ref{card_S4}) allows to conclude the proof of Lemma \ref{lem_ns}.

\end{proof}

\begin{proof}[Proof of Lemma \ref{caos_sup}] The proof follows from Lemma \ref{lem_s}, Lemma \ref{lem_ns} and (\ref{app6}).

\end{proof}

\subsection{Proof of Proposition \ref{prop-ho}}

\begin{proof}
It suffices to combine Lemma \ref{lem:3} and Lemma \ref{caos_sup} to get
\begin{eqnarray}\label{app11}
\text{\rm Var}(\text{\rm proj}(\mathcal L_0(n;u) | C_{\ge 4}) = O(1),
\end{eqnarray}
where the constant involved in the $O$-notation is absolute.

\end{proof}

\appendix

\section{EPC: technical lemmas}\label{app-approx}

By stationarity of the model the law of $\nabla f_n(x)$ is independent of $x\in \mathbb T$, it is centered Gaussian with covariance matrix given by $a_n$ in (\ref{an}). Since $\det(a_n) \ne 0$, Proposition 6.5 in \cite{azaiswschebor} (with $Z=\nabla f_n$) ensures that for every $z\in \mathbb R$,
\begin{equation}\label{last1}
\mathbb P(\exists x\in \mathbb T : \nabla f_n(x) = z, \det(\nabla^2 f_n(x)) = 0 ) =0.
\end{equation}
In particular for $z=0$, via a standard application of the inverse function theorem \cite[p.136]{adlertaylor}, we have that the set of critical points of $f_n$ a.s. consists of a finite number of isolated points. By B\'ezout Theorem we deduce that the number of critical points of $f_n$ is bounded from above by $4E_n$, and the Euler-Poincar\'e characteristic of any excursion set of $f_n$ so (see the Morse representation formula below).

Now, in order to apply area formula as in \cite[Proposition 6.1]{azaiswschebor} (with $f=\nabla f_n$) we need to be sure that the set of critical \emph{values} of $\nabla f_n$ a.s. has zero Lebeasgue measure. This follows from Sard's Lemma applied to $\nabla f_n$.

\begin{proof}[Proof of Lemma \ref{ubEPC}]
From (\ref{approx0}) we have
\begin{equation*}
|\mathcal L^\epsilon_0(u;n)| \le \int_{ \mathbb T } | \text{det}(\nabla^2 f_n(x))| \frac{1}{(2\epsilon)^2}\mathbf{1}_{[-\epsilon, \epsilon]^2}(\nabla f_n(x))\,dx,
\end{equation*}
where the random variable on the right hand side approximates the number of critical points of $ f_n$.
Thanks to the previous discussion, we can apply the area formula \cite[Proposition 6.1]{azaiswschebor} obtaining
\begin{equation}\label{areaformula}
 \int_{ \mathbb T } | \text{det}(\nabla^2 f_n(x))| \frac{1}{(2\epsilon)^2}\mathbf{1}_{[-\epsilon, \epsilon]^2}(\nabla f_n(x))\,dx = \frac{1}{(2\epsilon)^2}\int_{[-\epsilon, \epsilon]^2} \# \lbrace x\in \mathbb T : \nabla f_n(x) = z \rbrace\,dz.
 \end{equation}
By B\'ezout Theorem we have for every $z$
 \begin{equation}\label{bezout}
 \#\lbrace x\in \mathbb T : \nabla f_n(x) = z \rbrace \le 4 E_n.
 \end{equation}
Substituting (\ref{bezout}) into  (\ref{areaformula}) we obtain the desired result. \\
\end{proof}

By the Morse representation fomula \cite[\S9.3, \S9.4]{adlertaylor} we obtain
\begin{equation}
\mathcal{L}_{0}(n;u)=\sum_{j=0}^{2}(-1)^{j}\mu _{j}\left (A_u(f_n;\mathbb T), \left. f_{n}\right\vert _{{A_{u}(f_{n};\mathbb{%
T})}}\right ),
\label{morse}
\end{equation}
where
\begin{align*}
\mu _{j}\left (A_u(f_n;\mathbb T), \left. f_{n}\right\vert _{{A_{u}(f_{n};\mathbb{%
T})}}\right )& =\#\{x\in \mathbb{T}:f_{n}(x)\geq u,\nabla f_{n}(x)=0,\text{Ind}%
(-\nabla ^{2}f_{n}(x))=j\} \\
& =\#\{x\in \mathbb{T}:\Delta f_{n}(x)\leq
-E_{n}\,u,\nabla f_{n}(x)=0,\text{Ind}(-\nabla ^{2}f_{n}(x))=j\},
\end{align*}%
(note that in the last equality we used the fact that $\Delta f_n = - E_n f_n$)
$\text{Ind}(M)$ denoting the number of negative eigenvalues of a square
matrix $M$. More specifically, $\mu _{0}$ is the number of maxima, $\mu _{1}$
the number of saddles, and $\mu _{2}$ the number of minima in the excursion
region $A_{u}(f_{n};\mathbb{T})$.
Hence we can formally write
\begin{equation*}
\mathcal L_0(n;u) = \sum_{j=0}^2 (-1)^j\int_{\mathbb T}| \text{\rm det}(\nabla^2 f_n(x))| \mathbf{1}_{\lbrace \Delta f_{n}(x)\leq
-E_{n}\,u\rbrace} \mathbf{1}_{\lbrace \text{Ind}(-\nabla ^{2}f_{n}(x))=j \rbrace}\delta_0(\nabla f_n(x))\,dx
\end{equation*}
which is (\ref{ir0}).

\begin{proof}[Proof of Lemma \ref{approx-conv}]
Thanks to Morse representation formula and then Theorem 11.2.3 in \cite{adlertaylor} (whose assumptions are satisfied in particular thanks to (\ref{last1}) for $z=0$) we have a.s.
\begin{eqnarray*}
\lim_{\epsilon\to 0} \mathcal L_0^\epsilon(n;u) &=& \lim_{\epsilon\to 0}\sum_{j=0}^2 \frac{(-1)^j}{(2\epsilon)^2}\int_{\mathbb T}| \text{\rm det}(\nabla^2 f_n(x))| \mathbf{1}_{\lbrace \Delta f_{n}(x)\leq
-E_{n}\,u\rbrace} \mathbf{1}_{\lbrace \text{Ind}(-\nabla ^{2}f_{n}(x))=j \rbrace}\mathbf{1}_{[-\epsilon, \epsilon]^2}(\nabla f_n(x))\,dx\\
&=& \sum_{j=0}^2 (-1)^j \mu _{j}\left (A_u(f_n;\mathbb T), \left. f_{n}\right\vert _{{A_{u}(f_{n};\mathbb{%
T})}}\right ) = \mathcal L_0(n;u).
\end{eqnarray*}
The latter together with Lemma \ref{ubEPC} immediately establish the $L^2(\mathbb P)$-convergence thus concluding the proof.

\end{proof}

\section{Computation of covariance matrices}\label{Appendixcov}

Let $x,y\in \mathbb T$, and consider the Gaussian vector
$$
(\partial_1 f_n(x), \partial_2 f_n(x), \partial_1 f_n(y), \partial_2 f_n(y), \partial_{11} f_n(x), \partial_{12}f_n(x), \partial_{22}f_n(x), \partial_{11} f_n(y), \partial_{12}f_n(y), \partial_{22}f_n(y)).$$
It is convenient to write its covariance matrix in block-diagonal form,
i.e.
\begin{equation*}
\Sigma _{n}(x,y)=\left(
\begin{array}{cc}
A_{n}(x,y) & B_{n}(x,y) \\
B_{n}^{t}(x,y) & C_{n}(x,y)%
\end{array}%
\right).
\end{equation*}%
In particular the $A_{n}$ component collects the variances of the gradient
terms, and it is given by
\begin{equation*}
A_{n}(x,y)=\left(
\begin{array}{cc}
a_{n}(x,x) & a_{n}(x,y) \\
a_{n}(y,x) & a_{n}(y,y)%
\end{array}%
\right),\qquad a_{n}(x,y)=\left(
\begin{array}{cc}
r_{1,1}(x,y) & r_{1,2}(x,y) \\
r_{1,2}(x,y) & r_{2,2}(x,y)%
\end{array}%
\right) .
\end{equation*}%
It is easy to check that (cf. \cite{KKW, MPRW2015}), for $i=1,2$,
\begin{equation*}
r_{i,i}(x,y)=\frac{4\pi ^{2}}{\mathcal{N}_{n}}\;\sum_{\lambda }\lambda
_{(i)}^{2}e(\langle \lambda ,x-y\rangle ),
\end{equation*}%
while for $i\neq j$, $i,j=1,2$
\begin{equation*}
r_{i,j}(x,y)=\mathbb{E}[\partial _{i}f_{n}(x)\partial _{j}f_{n}(y)]=\frac{%
4\pi ^{2}}{\mathcal{N}_{n}}\;\sum_{\lambda }\lambda _{(i)}\lambda
_{(j)}e(\langle \lambda ,x-y\rangle )=r_{j,i}(x,y).
\end{equation*}%
The matrix $B_{n}$ collects the covariances between first and second order
derivatives, and is given by
\begin{equation*}
B_{n}(x,y)=\left(
\begin{array}{cc}
0 & b_{n}(x,y) \\
b_{n}(y,x) & 0%
\end{array}%
\right) ,
\end{equation*}%
where
\begin{equation*}
b_{n}(x,y)=\left(
\begin{array}{ccc}
r_{1,11}(x,y) & r_{1,12}(x,y) & r_{1,22}(x,y) \\
r_{2,11}(x,y) & r_{2,12}(x,y) & r_{2,22}(x,y)%
\end{array}%
\right) =-b_{n}(y,x).
\end{equation*}%
and
\begin{align*}
r_{1,11}(x,y)=\mathbb{E}[\partial _{1}f_{n}(x)\partial _{11}f_{n}(y)]& =%
\mathbb{E}[-\frac{2^{2}\pi ^{2}}{\sqrt{\mathcal{N}_{n}}}\sum_{\lambda \in
\Lambda _{n}}a_{\lambda }e(\langle \lambda ,y\rangle )\lambda
_{(1)}^{2}\times \frac{2\pi i}{\sqrt{\mathcal{N}_{n}}}\sum_{\lambda \in
\Lambda _{n}}a_{\lambda }e(\langle \lambda ,x\rangle )\lambda _{(1)}] \\
& =-\frac{8\pi ^{3}i}{\mathcal{N}_{n}}\mathbb{E}[\sum_{\lambda \in \Lambda
_{n}}a_{\lambda }e(\langle \lambda ,y\rangle )\lambda _{(1)}^{2}\times
\sum_{\lambda \in \Lambda _{n}}a_{\lambda }e(\langle \lambda ,x\rangle
)\lambda _{(1)}] \\
& =-\frac{8\pi ^{3}i}{\mathcal{N}_{n}}\sum_{\lambda ,\lambda ^{\prime }}%
\mathbb{E}[a_{\lambda }a_{\lambda ^{\prime }}]\;e(\langle \lambda
_{(1)}^{\prime 2}\;\lambda _{(1)}^{\prime } \\
& =-\frac{8\pi ^{3}i}{\mathcal{N}_{n}}\sum_{\lambda }\mathbb{E}[a_{\lambda
}a_{-\lambda }]\;e(\langle \lambda ,x\rangle )\;e(\langle -\lambda ,y\rangle
)\;(-\lambda _{(1)})^{2}\;\lambda _{(1)} \\
& =-\frac{8\pi ^{3}i}{\mathcal{N}_{n}}\;\sum_{\lambda }\lambda
_{(1)}^{3}e(\langle \lambda ,x-y\rangle ),
\end{align*}%
so that for $i=1,2$
\begin{equation*}
r_{i,ii}(x,y)=\mathbb{E}[\partial _{i}f_{n}(x)\partial _{ii}f_{n}(y)]=-\frac{%
8\pi ^{3}i}{\mathcal{N}_{n}}\;\sum_{\lambda }\lambda _{(i)}^{3}e(\langle
\lambda ,x-y\rangle ),
\end{equation*}%
and we notice that
\begin{align*}
r_{i,ii}(y,x)=\mathbb{E}[\partial _{i}f_{n}(y)\partial _{ii}f_{n}(x)]& =%
\mathbb{E}[-\frac{2^{2}\pi ^{2}}{\sqrt{\mathcal{N}_{n}}}\sum_{\lambda \in
\Lambda _{n}}a_{\lambda }e(\langle \lambda ,x\rangle )\lambda
_{(i)}^{2}\times \frac{2\pi i}{\sqrt{\mathcal{N}_{n}}}\sum_{\lambda \in
\Lambda _{n}}a_{\lambda }e(\langle \lambda ,y\rangle )\lambda _{(i)}] \\
& =-\frac{8\pi ^{3}i}{\mathcal{N}_{n}}\mathbb{E}[\sum_{\lambda \in \Lambda
_{n}}a_{\lambda }e(\langle \lambda ,x\rangle )\lambda _{(i)}^{2}\times
\sum_{\lambda \in \Lambda _{n}}a_{\lambda }e(\langle \lambda ,y\rangle
)\lambda _{(i)}] \\
& =-\frac{8\pi ^{3}i}{\mathcal{N}_{n}}\sum_{\lambda ,\lambda ^{\prime }}%
\mathbb{E}[a_{\lambda }a_{\lambda ^{\prime }}]\;e(\langle \lambda ,x\rangle
)\;e(\langle \lambda _{(i)}^{\prime 2}\;\lambda _{(i)}^{\prime } \\
& =-\frac{8\pi ^{3}i}{\mathcal{N}_{n}}\sum_{\lambda }\mathbb{E}[a_{\lambda
}a_{-\lambda }]\;e(\langle \lambda ,x\rangle )\;e(\langle -\lambda ,y\rangle
)\;\lambda _{(i)}^{2}\;(-\lambda _{(i)}) \\
& =\frac{8\pi ^{3}i}{\mathcal{N}_{n}}\;\sum_{\lambda }\lambda
_{(1)}^{3}e(\langle \lambda ,x-y\rangle ),
\end{align*}%
and in general for $i,j,k=1,2$
\begin{equation*}
r_{i,jk}(x,y)=-\frac{8\pi ^{3}i}{\mathcal{N}_{n}}\;\sum_{\lambda }\lambda
_{(i)}\lambda _{(j)}\lambda _{(k)}e(\langle \lambda ,x-y\rangle
)=-r_{i,jk}(y,x),
\end{equation*}%
\noindent so that
\begin{equation*}
b_{n}(x,y)=\left(
\begin{array}{ccc}
r_{1,11}(x-y) & r_{1,12}(x-y) & r_{1,22}(x-y) \\
r_{1,12}(x-y) & r_{1,22}(x-y) & r_{2,22}(x-y)%
\end{array}%
\right) .
\end{equation*}

\noindent Finally, for the matrix $C_{n}(x,y)$, we have
\begin{equation*}
C_{n}(x,y)=\left(
\begin{array}{cc}
c_{n}(x,x) & c_{n}(x,y) \\
c_{n}(y,x) & c_{n}(y,y)%
\end{array}%
\right) ,
\end{equation*}%
where of course $c_{n}(x,x)=c_{n}(y,y),$
\begin{equation*}
c_{n}(x,y)=\left(
\begin{array}{ccc}
r_{11,11}(x,y) & r_{11,12}(x,y) & r_{11,22}(x,y) \\
r_{12,11}(x,y) & r_{12,12}(x,y) & r_{12,22}(x,y) \\
r_{22,11}(x,y) & r_{22,12}(x,y) & r_{22,22}(x,y)%
\end{array}%
\right)
\end{equation*}%
and
\begin{align*}
r_{11,11}(x,y)=\mathbb{E}[\partial _{11}f_{n}(x)\partial _{11}f_{n}(y)]& =%
\mathbb{E}[-\frac{2^{2}\pi ^{2}}{\sqrt{\mathcal{N}_{n}}}\sum_{\lambda \in
\Lambda _{n}}a_{\lambda }e(\langle \lambda ,x\rangle )\lambda
_{(1)}^{2}\times -\frac{2^{2}\pi ^{2}}{\sqrt{\mathcal{N}_{n}}}\sum_{\lambda
\in \Lambda _{n}}a_{\lambda }e(\langle \lambda ,y\rangle )\lambda _{(1)}^{2}]
\\
& =\frac{2^{4}\pi ^{4}}{\mathcal{N}_{n}}\;\sum_{\lambda }\lambda
_{(1)}^{4}e(\langle \lambda ,x-y\rangle ),
\end{align*}%
or, more generally%
\begin{equation*}
r_{ij,kl}(x,y)=\frac{2^{4}\pi ^{4}}{\mathcal{N}_{n}}\;\sum_{\lambda }\lambda
_{(i)}\lambda _{(j)}\lambda _{(k)}\lambda _{(l)}e(\langle \lambda
,x-y\rangle ),
\end{equation*}%
so that
\begin{equation*}
c_n(x,y)=\left(
\begin{array}{ccc}
r_{11,11}(x-y) & r_{11,12}(x-y) & r_{11,22}(x-y) \\
r_{11,12}(x-y) & r_{11,22}(x-y) & r_{12,22}(x-y) \\
r_{11,22}(x-y) & r_{12,22}(x-y) & r_{22,22}(x-y)%
\end{array}
\right)=c_n(y,x).
\end{equation*}
To sum up, we have:
\begin{equation*}
\Sigma _{n}(x,y)=\Sigma _{n}(x-y)=\left(
\begin{array}{cc}
A_{n}(x-y) & B_{n}(x-y) \\
B_{n}^{t}(x-y) & C_{n}(x-y)%
\end{array}%
\right) .
\end{equation*}%
where
\begin{equation*}
A_{n}(x-y)=\left(
\begin{array}{cc}
a_{n} & a_{n}(x-y) \\
a_{n}(x-y) & a_{n}%
\end{array}%
\right) ,
\end{equation*}%
with
\begin{equation}\label{an}
a_{n}=\frac{E_{n}}{2}\left(
\begin{array}{cc}
1 & 0 \\
0 & 1%
\end{array}%
\right) ,\hspace{0.4cm}a_{n}(x-y)=\left(
\begin{array}{cc}
r_{1,1}(x-y) & r_{1,2}(x-y) \\
r_{1,2}(x-y) & r_{1,1}(x-y)%
\end{array}%
\right)
\end{equation}%
\begin{equation*}
r_{i,j}(x,y)=\frac{4\pi ^{2}}{\mathcal{N}_{n}}\;\sum_{\lambda }\lambda
_{(i)}\lambda _{(j)}e(\langle \lambda ,x-y\rangle ).
\end{equation*}%
Similarly
\begin{equation*}
B_{n}(x-y)=\left(
\begin{array}{cc}
0 & b_{n}(x-y) \\
-b_{n}(x-y) & 0%
\end{array}%
\right) ,
\end{equation*}%
where
\begin{equation*}
b_{n}(x-y)=\left(
\begin{array}{ccc}
r_{1,11}(x,y) & r_{1,12}(x,y) & r_{1,12}(x,y) \\
r_{1,12}(x,y) & r_{1,12}(x,y) & r_{1,11}(x,y)%
\end{array}%
\right) ,
\end{equation*}%
\begin{equation*}
r_{i,jk}(x,y)=-\frac{8\pi ^{3}i}{\mathcal{N}_{n}}\;\sum_{\lambda }\lambda
_{(i)}\lambda _{(j)}\lambda _{(k)}e(\langle \lambda ,x-y\rangle ).
\end{equation*}%
Likewise
\begin{equation*}
C_{n}(x-y)=\left(
\begin{array}{cc}
c_{n} & c_{n}(x-y) \\
c_{n}(x-y) & c_{n}%
\end{array}%
\right) ,
\end{equation*}%
where
\begin{equation}\label{cn}
c_{n}=\frac{E_{n}^{2}}{8}\left(
\begin{array}{ccc}
3+\hat{\mu}_{n}(4) & 0 & 1-\hat{\mu}_{n}(4) \\
0 & 1-\hat{\mu}_{n}(4) & 0 \\
1-\hat{\mu}_{n}(4) & 0 & 3+\hat{\mu}_{n}(4)%
\end{array}%
\right) ,
\end{equation}%
\begin{equation*}
c_{n}(x-y)=\left(
\begin{array}{ccc}
r_{11,11}(x-y) & r_{11,12}(x-y) & r_{11,22}(x-y) \\
r_{11,12}(x-y) & r_{11,22}(x-y) & r_{11,12}(x-y) \\
r_{11,22}(x-y) & r_{11,12}(x-y) & r_{11,11}(x-y)%
\end{array}%
\right) ,
\end{equation*}%
and%
\begin{equation*}
r_{ij,kl}(x,y)=\frac{2^{4}\pi ^{4}}{\mathcal{N}_{n}}\;\sum_{\lambda }\lambda
_{(i)}\lambda _{(j)}\lambda _{(k)}\lambda _{(l)}e(\langle \lambda
,x-y\rangle ).
\end{equation*}

\subsection{The special case $x=y$}\label{cov=}

The previous expressions are greatly simplified for $x=y;$ we apply here the
following lemma from \cite[\S 4.1]{MPRW2015}.

\begin{lemma}
For every $n\in S$, we have
\begin{equation*}
\frac{1}{n^{2}\mathcal{N}_{n}}\sum_{\lambda \in \Lambda _{n}}\lambda
_{(1)}^{4}=\frac{1}{n^{2}\mathcal{N}_{n}}\sum_{\lambda \in \Lambda
_{n}}\lambda _{(2)}^{4}=\frac{1}{8}(3+\hat{\mu}_{n}(4)),\text{ }\frac{1}{%
n^{2}\mathcal{N}_{n}}\sum_{\lambda \in \Lambda _{n}}\lambda
_{(1)}^{2}\lambda _{(2)}^{2}=\frac{1}{8}\left( 1-\hat{\mu}_{n}(4)\right) .
\end{equation*}
\end{lemma}

It is then immediate to check that, by symmetry (see \cite{RudnickWigman},
Lemma 2.3)

\begin{equation*}
\mathbb{E}[\partial _{1}f_{n}(x)\;\partial _{2}f_{n}(x)]=\frac{4\pi ^{2}}{%
\mathcal{N}_{n}}\;\sum_{\lambda }\lambda _{1}\lambda _{2}=0\text{ ;}
\end{equation*}%
similarly we have
\begin{equation*}
\mathbb{E}[\partial _{1}f_{n}(x)\;\partial _{1}f_{n}(x)]=\frac{4\pi ^{2}}{%
\mathcal{N}_{n}}\sum_{\lambda }\lambda _{1}^{2}=4\pi ^{2}\frac{n}{2}=2\pi
^{2}n.
\end{equation*}

\noindent On the other hand, for second order derivatives we have:
\begin{align*}
\mathbb{E}[\partial _{11}f_{n}(x)\;\partial _{11}f_{n}(x)]& =\frac{2^{4}\pi
^{4}}{\mathcal{N}_{n}}\sum_{\lambda \in \Lambda _{n}}\lambda
_{(1)}^{4}=2^{4}\pi ^{4}n^{2}\frac{1}{8}(3+\hat{\mu}_{n}(4)) \\
& =E_{n}^{2}\frac{1}{8}(3+\hat{\mu}_{n}(4)).
\end{align*}

\begin{equation*}
\mathbb{E}[\partial _{11}f_{n}(x)\;\partial _{12}f_{n}(x)]=\frac{2^{4}\pi
^{4}}{\mathcal{N}_{n}}\sum_{\lambda \in \Lambda _{n}}\lambda
_{(1)}^{3}\lambda _{(2)}=0.
\end{equation*}

\begin{align*}
\mathbb{E}[\partial _{11}f_{n}(x)\;\partial _{22}f_{n}(x)]& =\frac{2^{4}\pi
^{4}}{\mathcal{N}_{n}}\sum_{\lambda \in \Lambda _{n}}\lambda
_{(1)}^{2}\lambda _{(2)}^{2}=2^{4}\pi ^{4}n^{2}\frac{1}{8}(1-\hat{\mu}%
_{n}(4)) \\
& =E_{n}^{2}\frac{1}{8}(1-\hat{\mu}_{n}(4)).
\end{align*}

\begin{equation*}
\mathbb{E}[\partial _{12}f_{n}(x)\;\partial _{12}f_{n}(x)]=\frac{2^{4}\pi
^{4}}{\mathcal{N}_{n}}\sum_{\lambda \in \Lambda _{n}}\lambda
_{(1)}^{2}\lambda _{(2)}^{2}=E_{n}^{2}\frac{1}{8}(1-\hat{\mu}_{n}(4)).
\end{equation*}

\begin{equation*}
\mathbb{E}[\partial _{12}f_{n}(x)\;\partial _{22}f_{n}(x)]=\frac{2^{4}\pi
^{4}}{\mathcal{N}_{n}}\sum_{\lambda \in \Lambda _{n}}\lambda _{(1)}\lambda
_{(2)}^{3}=0.
\end{equation*}

\begin{equation*}
\mathbb{E}[\partial _{22}f_{n}(x)\;\partial _{22}f_{n}(x)]=\frac{2^{4}\pi
^{4}}{\mathcal{N}_{n}}\sum_{\lambda \in \Lambda _{n}}\lambda
_{(2)}^{4}=E_{n}^{2}\frac{1}{8}(3+\hat{\mu}_{n}(4)).
\end{equation*}

We have hence shown that the $5\times 5$ covariance matrix of the vector of
gradient and second derivatives is
\begin{equation*}
\sigma _{n}(x)=\left(
\begin{array}{cc}
a_{n}(x) & b_{n}(x) \\
b_{n}^{t}(x) & c_{n}(x)%
\end{array}%
\right) ,
\end{equation*}%
where
\begin{eqnarray*}
a_{n} &=&a_{n}(x)=\frac{E_{n}}{2}\left(
\begin{array}{cc}
1 & 0 \\
0 & 1%
\end{array}%
\right) ,\hspace{0.4cm}b_{n}(x)=\left(
\begin{array}{cc}
0 & 0 \\
0 & 0%
\end{array}%
\right) , \\
c_{n} &=&c_{n}(x)=\frac{E_{n}^{2}}{8}\left(
\begin{array}{ccc}
3+\hat{\mu}_{n}(4) & 0 & 1-\hat{\mu}_{n}(4) \\
0 & 1-\hat{\mu}_{n}(4) & 0 \\
1-\hat{\mu}_{n}(4) & 0 & 3+\hat{\mu}_{n}(4)%
\end{array}%
\right) .
\end{eqnarray*}

\section{Proof of Lemma \ref{expval}}\label{Appendixexp}

Throughout this paper, we will exploited some results in the number theory
literature (see \cite{KKW}) that we report here for completeness. Recall
once again that
\begin{align*}
\hat{\mu}_{n}(4)& :=\int_{\mathcal{S}^{1}}z^{4}d\mu _{n}(z)=\frac{1}{%
\mathcal{N}_{n}}\sum_{\lambda \in \Lambda _{n}}\int_{\mathcal{S}%
^{1}}z^{4}\delta _{\frac{\lambda }{\sqrt{n}}}(z)dz=\frac{1}{n^{2}\mathcal{N}%
_{n}}\sum_{\lambda \in \Lambda _{n}}\lambda ^{4}=\frac{1}{n^{2}\mathcal{N}%
_{n}}\sum_{\lambda \in \Lambda _{n}}(\lambda _{1}+i\lambda _{2})^{4} \\
& =\frac{1}{n^{2}\mathcal{N}_{n}}\sum_{\lambda \in \Lambda _{n}}(\lambda
_{1}^{4}+4i\lambda _{2}\lambda _{1}^{3}-6\lambda _{1}^{2}\lambda
_{2}^{2}-4i\lambda _{1}\lambda _{2}^{3}+\lambda _{2}^{4})\text{ .}
\end{align*}%
Now, since $\Lambda _{n}$ is invariant under the group $W_{2}$ of signed
permutations, consisting of coordinate permutations and sign-change of any
coordinate we have
\begin{align*}
\hat{\mu}_{n}(4)& =\frac{1}{n^{2}\mathcal{N}_{n}}\sum_{\lambda \in \Lambda
_{n}}(\lambda _{1}^{4}-6\lambda _{1}^{2}\lambda _{2}^{2}+\lambda _{2}^{4}) \\
& =\frac{1}{n^{2}\mathcal{N}_{n}}\sum_{\lambda \in \Lambda _{n}}(\lambda
_{1}^{2}+\lambda _{2}^{2})^{2}-\frac{8}{n^{2}\mathcal{N}_{n}}\sum_{\lambda
\in \Lambda _{n}}\lambda _{1}^{2}\lambda _{2}^{2}=1-\frac{8}{n^{2}\mathcal{N}%
_{n}}\sum_{\lambda \in \Lambda _{n}}\lambda _{1}^{2}\lambda _{2}^{2},
\end{align*}%
so that
\begin{equation*}
\frac{1}{n^{2}\mathcal{N}_{n}}\sum_{\lambda \in \Lambda _{n}}\lambda
_{1}^{2}\lambda _{2}^{2}=\frac{1}{8}\left( 1-\hat{\mu}_{n}(4)\right) .
\end{equation*}%
Moreover
\begin{equation*}
\frac{1}{n^{2}\mathcal{N}_{n}}\sum_{\lambda \in \Lambda _{n}}\lambda
_{1}^{4}=\frac{1}{n^{2}\mathcal{N}_{n}}\sum_{\lambda \in \Lambda
_{n}}\lambda _{2}^{4},
\end{equation*}%
and therefore
\begin{align*}
\frac{1}{n^{2}\mathcal{N}_{n}}\sum_{\lambda \in \Lambda _{n}}\lambda
_{1}^{4}& =\frac{1}{2n^{2}\mathcal{N}_{n}}\sum_{\lambda \in \Lambda
_{n}}(\lambda _{1}^{4}+\lambda _{2}^{4})=\frac{1}{2n^{2}\mathcal{N}_{n}}%
\sum_{\lambda \in \Lambda _{n}}(\lambda _{1}^{2}+\lambda _{2}^{2})^{2}-\frac{%
2}{2n^{2}\mathcal{N}_{n}}\sum_{\lambda \in \Lambda _{n}}\lambda
_{1}^{2}\lambda _{2}^{2} \\
& =\frac{1}{2}-\frac{1}{8}\left( 1-\hat{\mu}_{n}(4)\right) =\frac{1}{8}(3+%
\hat{\mu}_{n}(4)).
\end{align*}%
For $i,j=1,2$ with $i\neq j$, and $n,m=0,1,2,\dots $, since $\Lambda _{n}$
is invariant under the sign-change of any coordinate, we have
\begin{equation*}
\sum_{\lambda \in \Lambda _{n}}\lambda _{i}^{2n+1}\lambda
_{j}^{m}=\sum_{\lambda \in \Lambda _{n}}(-\lambda _{i})^{2n+1}\lambda
_{j}^{m}=0.
\end{equation*}%
\bigskip

\noindent Using invariance under $W_{d}$ in \cite[Lemma 2.3]{Rudnick} the
following lemma is proved :

\begin{lemma}
\label{grouplemma} For any subset $\mathcal{O} \subset \Lambda_n$ which is
invariant under the group $W_d$, we have
\begin{equation}  \label{RW}
\sum_{\lambda \in \mathcal{O} } \lambda_{(j)} \lambda_{(k)}=|\mathcal{O}|
\frac{n }{d} \delta_{j,k}.
\end{equation}
\end{lemma}

\noindent We note that using the invariance of $\Lambda _{n}$ under the
group $W_{d}$, we also immediately obtain that
\begin{equation*}
{\sum_{\lambda \in \Lambda _{n}}}\prod_{i=1}^{d}\lambda _{(i)}^{\alpha
_{i}}=0,
\end{equation*}%
if at least one of the exponents $\alpha _{i}$ is odd.

It is now possible to focus on the derivation of the expected values.
Actually the result for the excursion area is immediate and the result for
the boundary length was given already, for instance, in \cite{oravecz}, \cite%
{MPRW2015}. We can then focus on the EPC.

In general, a very powerful tool for the derivation of expected values of
Lipschitz-Killing Curvatures is provided by the Gaussian Kinematic Formula
(see \cite{adlertaylor}, Chapter 11), which was indeed exploited to derive
the analogous result in the case of random spherical harmonics. However
arithmetic random waves are not isotropic processes, which makes the
application of the GKF possible but more complicated; because of this, we
prefer to give here a proof from first principles.

\begin{proof}[Proof of Lemma \ref{expval}]
By Kac-Rice formula we can write
\begin{equation*}
\mathbb{E}\left[ \mathcal{L}_{0}(A_{u}(f_{n};\mathbb{T}))\right] =\int_{%
\mathbb{T}}K_{1}(x;I)dx
\end{equation*}%
where
\begin{equation*}
K_{1}(x;I)=K_{1;n}(x;I)=\phi _{\nabla f_{n}(x)}(\mathbf{0})\;\mathbb{E}[%
\text{det}H_{f_{n}}(x)\cdot \mathbb{I}_{I}(f_{n}(x))\;|\;\nabla f_{n}(x)=%
\mathbf{0}]\text{ ;}
\end{equation*}%
here we have
\begin{align*}
\phi _{\nabla f_{n}(x)}(\mathbf{0})& =\frac{1}{2\pi }\frac{2}{E_{n}}, \\
\mathbb{E}[\text{det}H_{f_{n}}(x)\cdot \mathbb{I}_{I}(f_{n}(x))\;|\;\nabla
f_{n}(x)=\mathbf{0}]& =\mathbb{E}[\text{det}H_{f_{n}}(x)\cdot \mathbb{I}_{I}(f_{n}(x))] \\
& =\frac{E_{n}^{2}}{8}\mathbb{E}\left[ (Z_{1}Z_{3}-Z_{2}^{2})\cdot \mathbb{I}%
_{\{\frac{Z_{1}+Z_{3}}{\sqrt{8}}\in I\}}\right]
\end{align*}%
where $(Z_{1},Z_{2},Z_{3})$ is a Gaussian vector with covariance matrix (see
Section \ref{Appendixcov})
\begin{equation*}
\left(
\begin{array}{ccc}
3+\hat{\mu}_{n}(4) & 0 & 1-\hat{\mu}_{n}(4) \\
0 & 1-\hat{\mu}_{n}(4) & 0 \\
1-\hat{\mu}_{n}(4) & 0 & 3+\hat{\mu}_{n}(4)%
\end{array}%
\right) .
\end{equation*}%
Now consider the transformation $W_{1}=Z_{1}$, $W_{2}=Z_{2}$, $%
W_{3}=Z_{1}+Z_{3}$, so that the vector $W$ is given by
\begin{equation*}
W=\left(
\begin{array}{ccc}
1 & 0 & 0 \\
0 & 1 & 0 \\
1 & 0 & 1%
\end{array}%
\right) Z
\end{equation*}%
with covariance matrix%
\begin{equation*}
\Sigma _{W}=\left(
\begin{array}{ccc}
1 & 0 & 0 \\
0 & 1 & 0 \\
1 & 0 & 1%
\end{array}%
\right) \left(
\begin{array}{ccc}
3+\hat{\mu}_{n}(4) & 0 & 1-\hat{\mu}_{n}(4) \\
0 & 1-\hat{\mu}_{n}(4) & 0 \\
1-\hat{\mu}_{n}(4) & 0 & 3+\hat{\mu}_{n}(4)%
\end{array}%
\right) \left(
\begin{array}{ccc}
1 & 0 & 1 \\
0 & 1 & 0 \\
0 & 0 & 1%
\end{array}%
\right)
\end{equation*}%
\begin{equation*}
=\left(
\begin{array}{ccc}
3+\hat{\mu}_{n}(4) & 0 & 4 \\
0 & 1-\hat{\mu}_{n}(4) & 0 \\
4 & 0 & 8%
\end{array}%
\right) .
\end{equation*}%
Under the obvious notation we write
\begin{equation*}
\Sigma _{(W_{1},W_{2})}=\left(
\begin{array}{cc}
3+\hat{\mu}_{n}(4) & 0 \\
0 & 1-\hat{\mu}_{n}(4)%
\end{array}%
\right) ,\hspace{0.5cm}\Sigma _{W_{3}}=8,
\end{equation*}%
so that the conditional distribution of $(W_{1},W_{2})|W_{3}=\sqrt{8}t$ is
Gaussian with covariance matrix
\begin{equation*}
\Sigma _{(W_{1},W_{2})|W_{3}}=\left(
\begin{array}{cc}
3+\hat{\mu}_{n}(4) & 0 \\
0 & 1-\hat{\mu}_{n}(4)%
\end{array}%
\right) -\left(
\begin{array}{c}
4 \\
0%
\end{array}%
\right) \frac{1}{8}\left(
\begin{array}{cc}
4 & 0%
\end{array}%
\right) =\left(
\begin{array}{cc}
1+\hat{\mu}_{n}(4) & 0 \\
0 & 1-\hat{\mu}_{n}(4)%
\end{array}%
\right) ,
\end{equation*}%
and expectation
\begin{equation*}
\mathbb{E}[(W_{1},W_{2})|W_{3}=\sqrt{8}t]=\left(
\begin{array}{c}
4 \\
0%
\end{array}%
\right) \frac{1}{8}\sqrt{8}t=\left(
\begin{array}{c}
\sqrt{2}t \\
0%
\end{array}%
\right) .
\end{equation*}%
We have that
\begin{equation*}
\mathbb{E}\left[ (Z_{1}Z_{3}-Z_{2}^{2})\cdot \mathbb{I}_{\{\frac{Z_{1}+Z_{3}%
}{\sqrt{8}}\in I\}}\right] =\mathbb{E}\left[ (W_{1}(W_{3}-W_{1})-W_{2}^{2})%
\cdot \mathbb{I}_{\{\frac{W_{3}}{\sqrt{8}}\in I\}}\right] .
\end{equation*}%
After the change of variable $\frac{w_{3}}{\sqrt{8}}=t$
\begin{align*}
& \mathbb{E}\left[ (W_{1}(W_{3}-W_{1})-W_{2}^{2})\cdot \mathbb{I}_{\{\frac{%
W_{3}}{\sqrt{8}}\in I\}}\right]  \\
& =\mathbb{E}_{(W_{1},W_{2})}\left[ \mathbb{E}\Big[%
(w_{1}(W_{3}-w_{1})-w_{2}^{2})\cdot \mathbb{I}_{\{\frac{W_{3}}{\sqrt{8}}\in
I\}}\;|(W_{1},W_{2})=(w_{1},w_{2})\Big]\right]  \\
& =\mathbb{E}_{(W_{1},W_{2})}\left[ \int_{\mathbb{R}%
}(w_{1}(w_{3}-w_{1})-w_{2}^{2})\cdot \mathbb{I}_{\{\frac{w_{3}}{\sqrt{8}}\in
I\}}\frac{1}{\sqrt{2\pi 8}}e^{-\frac{w_{3}^{2}}{2\cdot 8}}dw_{3}%
\;|(W_{1},W_{2})=(w_{1},w_{2})\right]  \\
& =\mathbb{E}_{(W_{1},W_{2})}\left[ \int_{\mathbb{R}}(w_{1}(t\sqrt{8}%
-w_{1})-w_{2}^{2})\cdot \mathbb{I}_{\{t\in I\}}\frac{1}{\sqrt{2\pi 8}}e^{-%
\frac{(t\sqrt{8})^{2}}{2\cdot 8}}\sqrt{8}dt\;|(W_{1},W_{2})=(w_{1},w_{2})%
\right]
\end{align*}%
\begin{align*}
& =\sqrt{8}\;\int_{\mathbb{R}}\mathbb{E}\left[ (W_{1}(t\sqrt{8}%
-W_{1})-W_{2}^{2})\right] \cdot \mathbb{I}_{\{t\in I\}}\frac{1}{\sqrt{2\pi 8}%
}e^{-\frac{(t\sqrt{8})^{2}}{2\cdot 8}}dt \\
& =\sqrt{8}\;\int_{I}\mathbb{E}\left[ (W_{1}(t\sqrt{8}-W_{1})-W_{2}^{2})%
\right] \frac{1}{\sqrt{2\pi 8}}e^{-\frac{(t\sqrt{8})^{2}}{2\cdot 8}}dt \\
& =\sqrt{8}\;\int_{I}\mathbb{E}\left[ (W_{1}(t\sqrt{8}-W_{1})-W_{2}^{2})%
\right] \phi _{W_{3}}(t\sqrt{8})dt \\
& =\sqrt{8}\;\int_{I}\mathbb{E}\left[ (W_{1}(W_{3}-W_{1})-W_{2}^{2})|W_{3}=t%
\sqrt{8}\right] \phi _{W_{3}}(t\sqrt{8})dt\text{ .}
\end{align*}%
Now note that
\begin{equation*}
\phi _{W_{3}}(\sqrt{8}\,t)=\frac{1}{4\sqrt{\pi }}e^{-\frac{t^{2}}{2}}
\end{equation*}%
and
\begin{align*}
& \mathbb{E}\left[ (W_{1}(W_{3}-W_{1})-W_{2}^{2})\Big|W_{3}=\sqrt{8}t\right]
\\
& =\mathbb{E}\left[ (W_{1}\sqrt{8}t-W_{1}^{2}-W_{2}^{2})\Big|W_{3}=\sqrt{8}t%
\right]  \\
& =\mathbb{E}\left[ ((X_{1}\sqrt{1+\hat{\mu}_{n}(4)}+\sqrt{2}\,t)\sqrt{8}%
\,t-(X_{1}\sqrt{1+\hat{\mu}_{n}(4)}+\sqrt{2}\,t)^{2}-X_{2}^{2}(1-\hat{\mu}%
_{n}(4)))\right]  \\
& =\mathbb{E}\left[ (2t^{2}-(X_{1}^{2}+X_{2}^{2})+\hat{\mu}%
_{n}(4)(X_{2}^{2}-X_{1}^{2}))\right]  \\
& =\mathbb{E}\left[ (2t^{2}-X_{1}^{2}(1+\hat{\mu}_{n}(4))-X_{2}^{2}(1-\hat{%
\mu}_{n}(4)))\right] ,
\end{align*}%
for $X_{1}$, $X_{2}$ standard independent Gaussian. Hence%
\begin{align*}
\mathbb{E}[\mathcal{L}_{0}(A_{I}(f_{n};\mathbb{T})))]& =\int_{\mathbb{T}}dx%
\frac{1}{\pi E_{n}}\frac{E_{n}^{2}}{8}\sqrt{8}\int_{I}\frac{1}{4\sqrt{\pi }}%
e^{-\frac{t^{2}}{2}}\mathbb{E}\left[ 2t^{2}-(X_{1}^{2}+X_{2}^{2})+\hat{\mu}%
_{n}(4)(X_{2}^{2}-X_{1}^{2})\right] dt \\
& =\frac{E_{n}}{8\pi }\frac{\sqrt{8}}{4\sqrt{\pi }}\int_{\mathbb{T}%
}dx\int_{I}e^{-\frac{t^{2}}{2}}\mathbb{E}\left[ 2t^{2}-(X_{1}^{2}+X_{2}^{2})+%
\hat{\mu}_{n}(4)(X_{2}^{2}-X_{1}^{2})\right] dt\text{ .} \\
& =\frac{E_{n}}{8\pi }\frac{\sqrt{8}}{4\sqrt{\pi }}\int_{I}e^{-\frac{t^{2}}{2%
}}\mathbb{E}\left[ 2t^{2}-(X_{1}^{2}+X_{2}^{2})+\hat{\mu}%
_{n}(4)(X_{2}^{2}-X_{1}^{2})\right] dt.
\end{align*}%
where we have exploited the fact that Area$({\mathbb{T}})=1.$ Writing $\Xi
=X_{1}^{2}+X_{2}^{2}$ and $\Theta :=X_{1}^{2}-X_{2}^{2}$ we
observe that $\mathbb{E}[\Xi ]=2,$ $\mathbb{E}[\Theta
]=0,$ so
\begin{align*}
\mathbb{E}[\mathcal{L}_{0}(A_{I}(f_{n};\mathbb{T})))]& =\frac{E_{n}}{8\pi }%
\frac{\sqrt{8}}{4\sqrt{\pi }}\int_{I}e^{-\frac{t^{2}}{2}}\mathbb{E}\left[
2t^{2}-\Xi +\hat{\mu}_{n}(4)\Theta \right] dt \\
& =\frac{E_{n}}{8\pi }\frac{\sqrt{8}}{4\sqrt{\pi }}\int_{I}e^{-\frac{t^{2}}{2%
}}\left[ 2t^{2}-2\right] dt.
\end{align*}%
For $I=(u,\infty )$
\begin{equation*}
\mathbb{E}[\mathcal{L}_{0}(A_{u}(f_{n};\mathbb{T})))]=\frac{E_{n}}{8\pi }%
\frac{\sqrt{8}}{4\sqrt{\pi }}\int_{u}^{\infty }e^{-\frac{t^{2}}{2}}\left[
2t^{2}-2\right] dt=\frac{E_{n}}{2\sqrt{8}\sqrt{\pi }\pi }ue^{-\frac{u^{2}}{2}%
},
\end{equation*}
which completes the proof.

\end{proof}

\section{EPC: second chaotic component \label{Appendix2chaos}}

\subsection{Proof of Proposition \protect\ref{proj} \label{proofproj}}

Let
\begin{equation*}
\alpha _{n}=\frac{k_{2}+k_{3}}{E}=\frac{\sqrt{2}}{\sqrt{3+\hat{\mu}_{n}(4)}},%
\hspace{0.5cm}\beta _{n}=\frac{k_{5}}{E}=\frac{\sqrt{1+\hat{\mu}_{n}(4)}}{%
\sqrt{3+\hat{\mu}_{n}(4)}},
\end{equation*}%
note that $\alpha _{n}^{2}=\frac{2}{3+\hat{\mu}_{n}(4)}$ and $\alpha
_{n}^{2}+\beta _{n}^{2}=1.$ Now let
\begin{equation*}
\varphi _{a}=\lim_{\varepsilon \rightarrow 0}\mathbb{E}[H_{a}(Y)\delta
_{\varepsilon }(\lambda _{1}\,Y)],\hspace{1cm}a=0,1,2.
\end{equation*}

\begin{equation*}
\theta_{ab}(u)=\mathbb{E}\left[ Y_{a}Y_{b}1\hspace{-0.27em}\mbox{\rm l}%
_{\left\{ \alpha_n Y_{3}+ \beta_n Y_{5}\leq -u\right\} }\right] ,\hspace{1cm}%
a,b=3,4,5,
\end{equation*}

\begin{equation*}
\psi _{abcd}(u)=\mathbb{E}\left[ Y_{a}Y_{b}Y_{c}Y_{d}1\hspace{-0.27em}%
\mbox{\rm l}_{\left\{ \alpha _{n}Y_{3}+\beta _{n}Y_{5}\leq -u\right\} }%
\right] ,\hspace{1cm}a,b,c,d=3,4,5.
\end{equation*}%
Note first that, as in \cite{CM2018}
\begin{equation}
\varphi _{a}(\ell )=%
\begin{cases}
\frac{1}{\sqrt{2\pi }k_{1}}, & a=0, \\
0, & a=1, \\
-\frac{1}{\sqrt{2\pi }k_{1}}, & a=2.%
\end{cases}
\label{immersione2}
\end{equation}

We hence immediately have $h_{1j}(u;n)=0$ for all $j>1$ and $h_{2j}(u;n)=0$
for all $j>2$ since $\varphi _{1}=0$. Moreover, by some straightforward but
tedious manipulations we obtain

\begin{equation*}
h_{34}(u;n)=[k_{3}k_{5}\;\psi _{3345}(u)+k_{2}k_{3}\;\psi
_{3334}(u)-k_{4}^{2}\;\psi _{3444}(u)]\varphi _{0}^{2}=0,
\end{equation*}

\begin{align*}
h_{35}(u;n)& =[k_{3}k_{5}\;\psi _{3355}(u)+k_{2}k_{3}\;\psi
_{3335}(u)-k_{4}^{2}\;\psi _{3445}(u)]\varphi _{0}^{2} \\
& =\frac{n\pi \sqrt{2}\sqrt{1+\hat{\mu}_{n}(4)}\left[ u\phi (u)(1+u^{2})+(3+%
\hat{\mu}_{n}(4))\Phi (-u)\right] }{3+\hat{\mu}_{n}(4)},
\end{align*}%
and moreover

\begin{equation*}
h_{45}(u;n)=[k_{3}k_{5}\;\psi _{3455}(u)+k_{2}k_{3}\;\psi
_{3345}(u)-k_{4}^{2}\;\psi _{4445}(u)]\varphi _{0}^{2}=0,
\end{equation*}

\begin{align*}
h_{1}(u;n )& =h_{2}(u;n)=[ k_{3} k_{5}\;\theta_{35}(u)+ k_{2} k_{3}\;\theta
_{33}(u) - k_{4}^{2}\;\theta_{44}(u)]\varphi_0 \varphi_2 \\
& = - n \pi u \,\phi(u),
\end{align*}

\begin{align*}
h_{3}(u; n )&=[k_{3} k_{5} \;\psi _{3335}(u)+ k_{2} k_{3}\;\psi _{3333}(u)-
k_{4}^{2}\;\psi _{3344}(u)]\varphi_0 ^{2} \\
& \;\;-[k_{3} k_{5}\;\theta _{35}(u)+ k_{2} k_{3}\;\theta
_{33}(u)-k_{4}^{2}\;\theta_{44}(u)]\varphi_0^{2} \\
&=n \pi \left[ \frac{2 u (1+u^2) \phi(u)}{3+\hat{\mu}_n(4) } + \Phi(-u) -
\hat{\mu}_n(4) \Phi(-u) \right] ,
\end{align*}

\begin{align*}
h_{4}(u;n )& =[k_{3} k_{5}\;\psi _{3445}(u)+k_{2}
k_{3}\;\psi_{3344}(u)-k_{4}^{2}\;\psi_{4444}(u)]\varphi_0^{2} \\
& \;\;-[k_{3} k_{5}\;\theta _{35}(u)+k_{2} k_{3}\;\theta _{33}(u)-
k_{4}^{2}\;\theta _{44}(u)]\varphi_0^{2} \\
& = -n \pi (1-\hat{\mu}_n(4)) \Phi (-u),
\end{align*}

\begin{align*}
h_{5}(u;n)& =[k_{3}k_{5}\;\psi _{3555}(u)+k_{2}k_{3}\;\psi
_{3355}(u)-k_{4}^{2}\;\psi _{4455}(u)]\varphi _{0}^{2} \\
& \;\;-[k_{3}k_{5}\;\theta _{35}(u)+k_{2}k_{3}\;\theta
_{33}(u)-k_{4}^{2}\;\theta _{44}(u)]\varphi _{0}^{2} \\
& =\frac{n\pi u(1+u^{2})(1+\hat{\mu}_{n}(4))\phi (u)}{3+\hat{\mu}_{n}(4)}%
\text{ ;}
\end{align*}%
In the previous steps, we have used a number of auxiliary functions $\psi
_{abcd}(u),$ for $a,b,c,d=3,4,5,$ whose exact expressions and derivations
are given in Lemmas \ref{theta} and \ref{psi} below.

\begin{lemma}
\label{theta}We have that
\begin{equation*}
\theta _{33}(u)=\Phi (-u)+u\,\phi (u)\frac{2}{3+\hat{\mu}_{n}(4)},\text{ }%
\theta _{35}(u)=u\,\phi (u)\frac{\sqrt{2}\sqrt{1+\hat{\mu}_{n}(4)}}{3+\hat{%
\mu}_{n}(4)},\text{ and }\theta _{44}(u)=\Phi (-u).
\end{equation*}
\end{lemma}

\begin{proof}
Let $X$, $Y$ and $Z$ be three independent standard Gaussian random
variables; with the same arguments as in \cite{CM2018}, Lemma 12, Lemma 13
and Lemma 14, we have
\begin{align*}
\theta _{33}(u)& =\mathbb{E}\left[ Y^{2}1\hspace{-0.27em}\mbox{\rm l}%
_{\left\{ \alpha _{n}Y+\beta _{n}X\leq -u\right\} }\right] =\int_{-\infty
}^{\infty }y^{2}\phi (y)\Phi \left( \frac{-u-\alpha _{n}y}{\beta _{n}}%
\right) dy \\
& =\Phi (-u)+\alpha _{n}^{2}u\phi (-u)
\end{align*}%
\begin{align*}
\theta _{35}(u)& =\mathbb{E}\left[ XY1\hspace{-0.27em}\mbox{\rm l}_{\left\{
\alpha _{n}Y+\beta _{n}X\leq -u\right\} }\right] =\int_{-\infty }^{\infty
}y\phi (y)dy\int_{-\infty }^{\frac{-u-\alpha _{n}y}{\beta _{n}}}x\phi (x)dx
\\
& =-\int_{-\infty }^{\infty }y\phi (y)\phi \left( \frac{-u-\alpha _{n}y}{%
\beta _{n}}\right) dy=\alpha _{n}\beta _{n}u\,\phi (-u)
\end{align*}%
\begin{equation*}
\theta _{44}(u)=\mathbb{E}\left[ Z^{2}1\hspace{-0.27em}\mbox{\rm l}_{\left\{
\alpha _{n}Y+\beta _{n}X\leq -u\right\} }\right] =\int_{-\infty }^{\infty
}\phi (y)\Phi \left( \frac{-u-\alpha _{n}y}{\beta _{n}}\right) dy=\Phi (-u).
\end{equation*}
\end{proof}

The next computations involve moments of four random variables and are hence
a bit more involved.

\begin{lemma}
\label{psi} We have that%
\begin{equation*}
\psi _{3333}(u)=3\Phi (-u)+4\,u\,\phi (u)\frac{6+u^{2}+3\hat{\mu}_{n}(4)}{(3+%
\hat{\mu}_{n}(4))^{2}},\text{ }\psi _{4444}(u)=3\Phi (-u),
\end{equation*}%
\begin{align*}
\psi _{3355}(u)& =\Phi (-u)+u\,\phi (u)\frac{3+\hat{\mu}_{n}(4)^{2}+2u^{2}(1+%
\hat{\mu}_{n}(4))}{(3+\hat{\mu}_{n}(4))^{2}}, \\
\psi _{3555}(u)& =u\,\phi (u)\sqrt{2}\frac{\sqrt{1+\hat{\mu}_{n}(4)}}{(3+%
\hat{\mu}_{n}(4))^{2}}(6+u^{2}(1+\hat{\mu}_{n}(4))), \\
\psi _{3335}(u)& =u\,\phi (u)\sqrt{2}\frac{\sqrt{1+\hat{\mu}_{n}(4)}}{(3+%
\hat{\mu}_{n}(4))^{2}}(3+2u^{2}+3\hat{\mu}_{n}(4)),
\end{align*}%
and moreover%
\begin{align*}
\psi _{3344}(u)& =\Phi (-u)+u\,\phi (u)\frac{2}{3+\hat{\mu}_{n}(4)},\text{ }%
\psi _{4455}(u)=\Phi (-u)+u\,\phi (u)\frac{1+\hat{\mu}_{n}(4)}{3+\hat{\mu}%
_{n}(4)}, \\
\psi _{3445}(u)& =u\,\phi (u)\sqrt{2}\frac{\sqrt{1+\hat{\mu}_{n}(4)}}{3+\hat{%
\mu}_{n}(4)}.
\end{align*}%
The following remaining terms are identically zero:
\begin{equation*}
\psi _{3334}(u)=\psi _{3345}(u)=\psi _{3444}(u)=\psi _{3455}(u)=\psi
_{4445}(u)=0.
\end{equation*}
\end{lemma}

\begin{proof}
In the sequel we shall use $X$, $Y$ and $Z$ to denote three independent
standard Gaussian random variables. The computations to follow are then just
standard evaluations of Gaussian integrals. In particular, applying \cite%
{CM2018} Lemma 12, Lemma 13 and Lemma 14, we have
\begin{align*}
\psi _{3333}(u)& =\mathbb{E}\left[ Y^{4}1\hspace{-0.27em}\mbox{\rm l}%
_{\left\{ \alpha _{n}Y+\beta _{n}X\leq -u\right\} }\right] =\int_{-\infty
}^{\infty }y^{4}\phi (y)\Phi \Big(\frac{-u-\alpha _{n}y}{\beta _{n}}\Big)dy
\\
& =3\Phi (-u)+u\,\phi (-u)\left[ 3\alpha _{n}^{2}+3\alpha _{n}^{4}\beta
_{n}^{2}+3\beta _{n}^{4}\alpha _{n}^{2}+\alpha _{n}^{4}u^{2}\right] .
\end{align*}%
\begin{equation*}
\psi _{4444}(u)=\mathbb{E}\left[ Z^{4}1\hspace{-0.27em}\mbox{\rm l}_{\left\{
\alpha _{n}Y+\beta _{n}X\leq -u\right\} }\right] =3\,\mathbb{E}\left[ 1%
\hspace{-0.27em}\mbox{\rm l}_{\left\{ \alpha _{n}Y+\beta _{n}X\leq
-u\right\} }\right] =3\Phi (-u).
\end{equation*}%
Now we observe that
\begin{equation*}
\int_{-\infty }^{q}x^{2}\phi (x)dx=\Phi (q)-q\,\phi (q),
\end{equation*}%
and we obtain
\begin{align*}
\psi _{3355}(u)& =\mathbb{E}\left[ Y^{2}X^{2}1\hspace{-0.27em}\mbox{\rm l}%
_{\left\{ \alpha _{n}Y+\beta _{n}X\leq -u\right\} }\right] =\int_{-\infty
}^{\infty }y^{2}\phi (y)dy\int_{-\infty }^{\frac{-u-\alpha _{n}y}{\beta _{n}}%
}x^{2}\phi (x)dx \\
& =\int_{-\infty }^{\infty }y^{2}\phi (y)\Phi \left( \frac{-u-\alpha _{n}y}{%
\beta _{n}}\right) dy-\left( \frac{-u-\alpha _{n}y}{\beta _{n}}\right)
\int_{-\infty }^{\infty }y^{2}\phi (y)\phi \left( \frac{-u-\alpha _{n}y}{%
\beta _{n}}\right) dy \\
& =\Phi (-u)+\alpha _{n}^{2}\,u\,\phi (-u)+\beta _{n}^{2}\,u\,\phi
(-u)(-2\alpha _{n}^{4}+\beta _{n}^{4}-\alpha _{n}^{2}\beta _{n}^{2}+\alpha
_{n}^{2}u^{2}).
\end{align*}%
Likewise%
\begin{align*}
\psi _{3555}(u)& =\mathbb{E}\left[ YX^{3}1\hspace{-0.27em}\mbox{\rm l}%
_{\left\{ \alpha _{n}Y+\beta _{n}X\right\} \leq -u}\right] =\int_{-\infty
}^{\infty }y\phi (y)dy\int_{-\infty }^{\frac{-u-\alpha _{n}y}{\beta _{n}}%
}x^{3}\phi (x)dx \\
& =-\int_{-\infty }^{\infty }y\phi (y)\phi \left( \frac{-u-\alpha _{n}y}{%
\beta _{n}}\right) \left\{ \left( \frac{-u-\alpha _{n}y}{\beta _{n}}\right)
^{2}+2\right\} dy=\alpha _{n}\beta _{n}\,u\,\phi (-u)(3\alpha _{n}^{2}+\beta
_{n}^{2}u^{2}),
\end{align*}%
\begin{equation*}
\psi _{3335}(u)=\mathbb{E}\left[ Y^{3}X1\hspace{-0.27em}\mbox{\rm l}%
_{\left\{ \alpha _{n}Y+\beta _{n}X\leq -u\right\} }\right] =\alpha _{n}\beta
_{n}u\,\phi (-u)(3\beta _{n}^{2}+\alpha _{n}^{2}u^{2}).
\end{equation*}%
Moreover%
\begin{equation*}
\psi _{3344}(u)=\mathbb{E}\left[ Z^{2}Y^{2}1\hspace{-0.27em}\mbox{\rm l}%
_{\left\{ \alpha _{n}Y+\beta _{n}X\leq -u\right\} }\right] =\mathbb{E}\left[
Y^{2}1\hspace{-0.27em}\mbox{\rm l}_{\left\{ \alpha _{n}Y+\beta _{\ell }X\leq
-u\right\} }\right] =\theta _{33}(u),
\end{equation*}%
\begin{align*}
\psi _{4455}(u)& =\mathbb{E}\left[ Z^{2}X^{2}1\hspace{-0.27em}\mbox{\rm l}%
_{\left\{ \alpha _{n}Y+\beta _{n}X\leq -u\right\} }\right] =\mathbb{E}\left[
X^{2}1\hspace{-0.27em}\mbox{\rm l}_{\left\{ \alpha _{n}Y+\beta _{\ell }X\leq
-u\right\} }\right] =\theta _{55}(u) \\
& =\Phi (-u)+\beta _{n}^{2}\,u\,\phi (-u),
\end{align*}%
and finally%
\begin{equation*}
\psi _{3445}(u)=\mathbb{E}\left[ Z^{2}XY1\hspace{-0.27em}\mbox{\rm l}%
_{\left\{ \alpha _{n}Y+\beta _{n}X\leq -u\right\} }\right] =\mathbb{E}\left[
XY1\hspace{-0.27em}\mbox{\rm l}_{\left\{ \alpha _{n}Y+\beta _{\ell }X\leq
-u\right\} }\right] =\theta _{35}(u).
\end{equation*}%
The fact that $\psi _{3334}(u)$, $\psi _{3345}(u)$, $\psi _{3444}(u)$, $\psi
_{3455}(u)$ and $\psi _{4445}(u)$ are identically equal to zero, it is
enough to note that they are all of the form
\begin{equation*}
\mathbb{E}\left[ Z^{p}X^{q}Y^{r}1\hspace{-0.27em}\mbox{\rm l}_{\left\{
\alpha _{n}Y+\beta _{n}X\leq -u\right\} }\right]
\end{equation*}%
where $p=1,3$ is odd.
\end{proof}

\subsection{Proof of Proposition \protect\ref{AB} \label{proofAB}}

We have to deal with the following integrals of squares:
\begin{equation}
I_{00}(n)=\int_{\mathbb{T}}f_{n}^{2}(x)dx,\text{ }I_{11}(n)=\int_{\mathbb{T}%
}\left\{ e_{1}^{x}f_{n}(x)\right\} ^{2}dx,\text{ }I_{22}(n)=\int_{\mathbb{T}%
}\left\{ e_{2}^{x}f_{n}(x)\right\} ^{2}dx;  \label{3marzo}
\end{equation}%
we shall also study the cross-product integral
\begin{equation*}
I_{0,22}(n)=\int_{\mathbb{T}}f_{n}(x)e_{2}^{x}e_{2}^{x}f_{n}(x)dx,
\end{equation*}%
and finally we shall consider
\begin{equation*}
I_{12,12}(n)=\int_{\mathbb{T}}\left\{ e_{1}^{x}e_{2}^{x}f_{n}(x)\right\}
^{2}dx,\text{ }I_{22,22}(n)=\int_{\mathbb{T}}\{e_{2}^{x}e_{2}^{x}f_{n}(x)%
\}^{2}dx.
\end{equation*}%
We have%
\begin{equation*}
A_{35}=-\frac{E_{n}}{k_{3}k_{5}}\left\{ 1+2\frac{k_{2}}{k_{3}}\right\}
I_{0,22}(n)-\frac{E_{n}^{2}k_{2}}{k_{3}^{2}k_{5}}I_{00}(n)-\frac{1}{%
k_{3}k_{5}}\left\{ 1+\frac{k_{2}}{k_{3}}\right\} I_{22,22}(n),
\end{equation*}
\begin{equation*}
B_{1}=\frac{1}{k_{1}^{2}}I_{11}(n)-1\text{ , }B_{2}=\frac{1}{k_{1}^{2}}%
I_{22}(n)-1,B_{3}=\frac{E_{n}^{2}}{k_{3}^{2}}I_{00}(n)+\frac{1}{k_{3}^{2}}%
I_{22,22}(n)+\frac{2E_{n}}{k_{3}^{2}}I_{0,22}(n)-1,
\end{equation*}%
\begin{equation*}
B_{4}=\frac{1}{k_{4}^{2}}I_{12,12}(n)-1,
\end{equation*}%
\begin{equation*}
B_{5}=\frac{1}{k_{5}^{2}}\big(1+\frac{k_{2}}{k_{3}}\big)^{2}I_{22,22}(n)+%
\frac{E_{n}^{2}k_{2}^{2}}{k_{3}^{2}k_{5}^{2}}I_{00}(n)+2\frac{E_{n}k_{2}}{%
k_{3}k_{5}^{2}}\big(1+\frac{k_{2}}{k_{3}}\big)I_{0,22}(n)-1.
\end{equation*}%
Our next step is then to investigate the behaviour of these integrals of
stochastic processes; this task is accomplished in the following Lemma.

\begin{lemma}
The following identities hold:%
\begin{align*}
I_{00}(n)& =\frac{1}{\mathcal{N}_{n}}\sum_{\lambda }|a_{\lambda }|^{2},\text{
}I_{11}(n)=\frac{4\pi ^{2}}{\mathcal{N}_{n}}\sum_{\lambda }|a_{\lambda
}|^{2}\lambda _{1}^{2},\text{ }I_{22}(n)=\frac{4\pi ^{2}}{\mathcal{N}_{n}}%
\sum_{\lambda }|a_{\lambda }|^{2}\lambda _{2}^{2}, \\
I_{0,22}(n)& =-\frac{4\pi ^{2}}{\mathcal{N}_{n}}\sum_{\lambda }|a_{\lambda
}|^{2}{\lambda _{2}^{2}},\text{ }I_{12,12}(n)=\frac{16\pi ^{4}}{\mathcal{N}%
_{n}}\sum_{\lambda }|a_{\lambda }|^{2}\lambda _{1}^{2}\lambda _{2}^{2},\text{
}I_{22,22}(n)=\frac{16\pi ^{4}}{\mathcal{N}_{n}}\sum_{\lambda }|a_{\lambda
}|^{2}\lambda _{2}^{4}.
\end{align*}
\end{lemma}

\begin{proof}
By Parseval's identity it follows that%
\begin{equation*}
I_{00}(n)=\frac{1}{\mathcal{N}_{n}}\sum_{\lambda }|a_{\lambda }|^{2},
\end{equation*}%
and similarly%
\begin{eqnarray*}
I_{11}(n) &=&-\frac{4\pi ^{2}}{\mathcal{N}_{n}}\sum_{\lambda ,\lambda
^{\prime }}a_{\lambda }a_{\lambda ^{\prime }}\lambda _{1}\lambda
_{1}^{\prime }\int_{\mathbb{T}}e(\langle \lambda ,x\rangle )e(\langle
\lambda ^{\prime },x\rangle )dx=\frac{4\pi ^{2}}{\mathcal{N}_{n}}%
\sum_{\lambda }|a_{\lambda }|^{2}\lambda _{1}^{2}\text{ ,} \\
I_{22}(n) &=&\int_{\mathbb{T}}\left\{ e_{2}^{x}f_{n}(x)\right\} ^{2}dx=\frac{%
4\pi ^{2}}{\mathcal{N}_{n}}\sum_{\lambda }|a_{\lambda }|^{2}\lambda _{2}^{2}%
\text{ }.
\end{eqnarray*}%
Likewise%
\begin{equation*}
I_{0,22}(n)=\frac{1}{\mathcal{N}_{n}}\sum_{\lambda ,\lambda ^{\prime
}}a_{\lambda }a_{\lambda ^{\prime }}\int_{\mathbb{T}}e(\langle \lambda
,x\rangle )(2\pi i\lambda _{2}^{\prime })^{2}e(\langle \lambda ^{\prime
},x\rangle )dx=-\frac{4\pi ^{2}}{\mathcal{N}_{n}}\sum_{\lambda }|a_{\lambda
}|^{2}\lambda _{2}^{2}\text{ ,}
\end{equation*}%
and finally%
\begin{equation*}
I_{12,12}(n)=\frac{16\pi ^{4}}{\mathcal{N}_{n}}\sum_{\lambda ,\lambda
^{\prime }}a_{\lambda }a_{\lambda ^{\prime }}\lambda _{1}\lambda _{2}\lambda
_{1}^{\prime }\lambda _{2}^{\prime }\int_{\mathbb{T}}e(\langle \lambda
,x\rangle )e(\langle \lambda ^{\prime },x\rangle )dx=\frac{16\pi ^{4}}{%
\mathcal{N}_{n}}\sum_{\lambda }|a_{\lambda }|^{2}\lambda _{1}^{2}\lambda
_{2}^{2}
\end{equation*}%
\begin{equation*}
I_{22,22}(n)=\frac{16\pi ^{4}}{\mathcal{N}_{n}}\sum_{\lambda ,\lambda
^{\prime }}a_{\lambda }a_{\lambda ^{\prime }}\lambda _{2}^{2}{\lambda
_{2}^{\prime }}^{2}\int_{\mathbb{T}}e(\langle \lambda ,x\rangle )e(\langle
\lambda ^{\prime },x\rangle )dx=\frac{16\pi ^{4}}{\mathcal{N}_{n}}%
\sum_{\lambda }|a_{\lambda }|^{2}\lambda _{2}^{4}.
\end{equation*}
\end{proof}

We are now in the position to complete the proof.

\begin{proof}[Proof of Proposition \protect\ref{AB}]
Note that

\begin{align*}
A_{35}(n)& =-\frac{E_{n}}{k_{3}k_{5}}\left\{ 1+2\frac{k_{2}}{k_{3}}\right\}
I_{0,22}(n)-\frac{E_{n}^{2}k_{2}}{k_{3}^{2}k_{5}}I_{00}(n)-\frac{1}{%
k_{3}k_{5}}\left\{ 1+\frac{k_{2}}{k_{3}}\right\} I_{22,22}(n) \\
& =-\frac{1}{\sqrt{2}\pi ^{2}n\sqrt{1+\hat{\mu}_{n}(4)}}\frac{5-\hat{\mu}%
_{n}(4)}{{3+\hat{\mu}_{n}(4)}}I_{0,22}(n)-\frac{1-\hat{\mu}_{n}(4)}{{3+\hat{%
\mu}_{n}(4)}}\frac{2\sqrt{2}}{\sqrt{1+\hat{\mu}_{n}(4)}}I_{00}(n) \\
& \;\;-\frac{1}{\sqrt{2}\pi ^{4}n^{2}\sqrt{1+\hat{\mu}_{n}(4)}}\frac{1}{{3+%
\hat{\mu}_{n}(4)}}I_{22,22}(n) \\
& =\frac{1}{\mathcal{N}_{n}}\sum_{\lambda }|a_{\lambda }|^{2}\left[ \frac{%
4\pi ^{2}}{\sqrt{2}\pi ^{2}n\sqrt{1+\hat{\mu}_{n}(4)}}\frac{5-\hat{\mu}%
_{n}(4)}{{3+\hat{\mu}_{n}(4)}}\lambda _{2}^{2}-\frac{1-\hat{\mu}_{n}(4)}{{3+%
\hat{\mu}_{n}(4)}}\frac{2\sqrt{2}}{\sqrt{1+\hat{\mu}_{n}(4)}}\right. \\
& \;\;\left. -\frac{16\pi ^{4}}{\sqrt{2}\pi ^{4}n^{2}\sqrt{1+\hat{\mu}_{n}(4)%
}}\frac{1}{{3+\hat{\mu}_{n}(4)}}\lambda _{2}^{4}\right] \\
& =\frac{1}{\mathcal{N}_{n}}\sum_{\lambda }|a_{\lambda }|^{2}\left[ \frac{2%
\sqrt{2}}{n\sqrt{1+\hat{\mu}_{n}(4)}}\frac{5-\hat{\mu}_{n}(4)}{{3+\hat{\mu}%
_{n}(4)}}\lambda _{2}^{2}-\frac{1-\hat{\mu}_{n}(4)}{{3+\hat{\mu}_{n}(4)}}%
\frac{2\sqrt{2}}{\sqrt{1+\hat{\mu}_{n}(4)}}\right. \\
& \;\;\left. -\frac{8\sqrt{2}}{n^{2}\sqrt{1+\hat{\mu}_{n}(4)}}\frac{1}{{3+%
\hat{\mu}_{n}(4)}}\lambda _{2}^{4}\right] .
\end{align*}

On the other hand%
\begin{equation*}
B_{1}=\frac{1}{2n\pi ^{2}}\frac{4\pi ^{2}}{\mathcal{N}_{n}}\sum_{\lambda
}|a_{\lambda }|^{2}\lambda _{1}^{2}-1=\frac{1}{\mathcal{N}_{n}}\sum_{\lambda
}|a_{\lambda }|^{2}\frac{2}{n}\lambda _{1}^{2}-1
\end{equation*}

\begin{equation*}
B_{2}=\frac{1}{\mathcal{N}_{n}}\sum_{\lambda }|a_{\lambda }|^{2}\frac{2}{n}%
\lambda _{2}^{2}-1
\end{equation*}

\begin{align*}
B_{3}& =\frac{E_{n}^{2}}{k_{3}^{2}}I_{00}(n)+\frac{1}{k_{3}^{2}}I_{22,22}(n)+%
\frac{2E_{n}}{k_{3}^{2}}I_{0,22}(n)-1 \\
& =\frac{8}{3+\hat{\mu}_{n}(4)}\frac{1}{\mathcal{N}_{n}}\sum_{\lambda
}|a_{\lambda }|^{2}+\frac{1}{2\pi ^{4}n^{2}(3+\hat{\mu}_{n}(4))}\frac{16\pi
^{4}}{\mathcal{N}_{n}}\sum_{\lambda }|a_{\lambda }|^{2}\lambda _{2}^{4}-%
\frac{4}{\pi ^{2}n(3+\hat{\mu}_{n}(4))}\frac{4\pi ^{2}}{\mathcal{N}_{n}}%
\sum_{\lambda }|a_{\lambda }|^{2}{\lambda _{2}^{2}}-1 \\
& =\frac{1}{\mathcal{N}_{n}}\sum_{\lambda }|a_{\lambda }|^{2}\left[ \frac{8}{%
3+\hat{\mu}_{n}(4)}+\frac{8}{n^{2}(3+\hat{\mu}_{n}(4))}\lambda _{2}^{4}-%
\frac{16}{n(3+\hat{\mu}_{n}(4))}{\lambda _{2}^{2}}\right] -1
\end{align*}

\begin{align*}
B_{4}& =\frac{1}{k_{4}^{2}}I_{12,12}(n)-1 \\
& =\frac{1}{2\pi ^{4}n^{2}(1-\hat{\mu}_{n}(4))}\frac{16\pi ^{4}}{\mathcal{N}%
_{n}}\sum_{\lambda }|a_{\lambda }|^{2}\lambda _{1}^{2}\lambda _{2}^{2}-1 \\
& =\frac{8}{n^{2}(1-\hat{\mu}_{n}(4))}\frac{1}{\mathcal{N}_{n}}\sum_{\lambda
}|a_{\lambda }|^{2}\lambda _{1}^{2}\lambda _{2}^{2}-1.
\end{align*}

\begin{align*}
B_{5}& =\frac{1}{k_{5}^{2}}\big(1+\frac{k_{2}}{k_{3}}\big)^{2}I_{22,22}(n)+%
\frac{E_{n}^{2}k_{2}^{2}}{k_{3}^{2}k_{5}^{2}}I_{00}(n)+2\frac{E_{n}k_{2}}{%
k_{3}k_{5}^{2}}\big(1+\frac{k_{2}}{k_{3}}\big)I_{0,22}(n)-1 \\
& =\frac{1}{\pi ^{4}n^{2}(1+\hat{\mu}_{n}(4))(3+\hat{\mu}_{n}(4))}\frac{%
16\pi ^{4}}{\mathcal{N}_{n}}\sum_{\lambda }|a_{\lambda }|^{2}\lambda
_{2}^{4}+\frac{(1-\hat{\mu}_{n}(4))^{2}}{(3+\hat{\mu}_{n}(4))(1+\hat{\mu}%
_{n}(4))}\frac{1}{\mathcal{N}_{n}}\sum_{\lambda }|a_{\lambda }|^{2} \\
& \;\;-\frac{2(1-\hat{\mu}_{n}(4))}{\pi ^{2}n(3+\hat{\mu}_{n}(4))(1+\hat{\mu}%
_{n}(4))}\frac{4\pi ^{2}}{\mathcal{N}_{n}}\sum_{\lambda }|a_{\lambda }|^{2}{%
\lambda _{2}^{2}}-1 \\
& =\frac{1}{\mathcal{N}_{n}}\sum_{\lambda }|a_{\lambda }|^{2}\Big[\frac{16}{%
n^{2}(1+\hat{\mu}_{n}(4))(3+\hat{\mu}_{n}(4))}\lambda _{2}^{4}+\frac{(1-\hat{%
\mu}_{n}(4))^{2}}{(3+\hat{\mu}_{n}(4))(1+\hat{\mu}_{n}(4))} \\
& \;\;-\frac{2(1-\hat{\mu}_{n}(4))}{n(3+\hat{\mu}_{n}(4))(1+\hat{\mu}_{n}(4))%
}4{\lambda _{2}^{2}}\Big]-1.
\end{align*}
which concludes the proof.

\end{proof}

\end{document}